\documentclass[a4paper,10pt]{amsart}

\usepackage[T1]{fontenc}
\usepackage{graphicx}
\usepackage{amsmath}
\usepackage{amssymb,amsthm}
\usepackage{mathrsfs}
\usepackage{comment}
\usepackage{xcolor}
\usepackage{hyperref}
\usepackage{amsaddr}
\usepackage{tikz}
\usetikzlibrary{arrows}

\newtheorem{thm}{Theorem}
\newtheorem{prop}[thm]{Proposition}
\newtheorem{lem}[thm]{Lemma}
\newtheorem{cor}[thm]{Corollary}
\theoremstyle{definition}
\newtheorem{Def}[thm]{Definition}
\newtheorem{Not}[thm]{Notation}
\theoremstyle{remark}
\newtheorem{rem}[thm]{Remark}
\newtheorem{exmp}[thm]{Example}

\numberwithin{equation}{section}
\numberwithin{thm}{section}

\newcommand{\ds}{\displaystyle}

\newenvironment{proofof}[1]{\medskip\noindent
   \textbf{Proof of #1:} }{\hfill $\square$\par\medskip}
\DeclareRobustCommand{\stirling}{\genfrac\{\}{0pt}{}}

\title[Multivariate Juggling Probabilities]{
  Multivariate Juggling Probabilities}
\author[A. Ayyer, J. Bouttier, S. Corteel, and F. Nunzi]{Arvind Ayyer$^1$, J\'er\'emie Bouttier$^{2,3}$, Sylvie Corteel$^4$ and Fran\c{c}ois Nunzi$^4$}
\date{\today}
\email{arvind@math.iisc.ernet.in}
\email{jeremie.bouttier@cea.fr}
\email{corteel@liafa.univ-paris-diderot.fr}
\email{fnunzi@liafa.univ-paris-diderot.fr}

\address{$^1$ Department of Mathematics, Indian Institute of Science,\\ Bangalore - 560012, India}
\address{$^2$ Institut de Physique Th\'eorique, CEA, IPhT, 91191 Gif-sur-Yvette, France, CNRS URA 2306}
\address{$^3$ D\'epartement de Math\'ematiques et Applications, \'Ecole normale sup\'erieure,\\ 45 rue d'Ulm, F-75231 Paris Cedex 05}
\address{$^4$ LIAFA, CNRS et Universit\'e Paris Diderot, Case 7014, F-75205 Paris Cedex 13}

\thanks{The authors are partially funded by ANR grants ANR-08-JCJC-0011 and 
ANR 12-JS02-001-01 and by the projet \'Emergences Combinatoire \`a Paris}


\begin{document}

\begin{abstract}
We consider refined versions of Markov chains related to juggling introduced by Warrington. We further generalize the construction to juggling with 
arbitrary heights as well as infinitely many balls, which are expressed more succinctly in terms of Markov chains on integer partitions. 
In all cases, we give explicit product formulas for the stationary probabilities. The normalization factor in one case can be explicitly written as a homogeneous symmetric polynomial.
We also refine and generalize enriched Markov chains on set partitions. Lastly, we prove that in one case, the stationary
distribution is attained in bounded time.
\end{abstract}

\maketitle

\section{Introduction}
Although juggling as a human endeavour has been around since time
immemorial, it is fairly recently that mathematicians have taken an
active interest in exploring the field. Combinatorialists became
interested in juggling towards the end of the last century after an
article in the \emph{Amer. Math. Monthly} by Buhler, Eisenbud, Graham
and Wright \cite{Buhleretal}, where they enumerate what they call {\em
  juggling sequences} and relate it to other known combinatorial
structures.  Since then, their results have been $q$-ified
\cite{EhrenborgReaddy} and further refined in various ways
\cite{Stadler,Stadler2,ChungGraham1,ChungGraham2,ButlerGraham}. Other
connections between juggling and mathematics appear for instance in
algebraic geometry \cite{DevadossMugno,KnutsonLamSpeyer}.  A
mathematical history of juggling is given in the fascinating book by
Polster \cite{Polster}.

Probabilists, on the other hand, are newer to this game, possibly
because no one is able to perform totally random juggling. Coincidentally, it
was another popular article in the \emph{Amer. Math. Monthly} by
Warrington \cite{Warrington} on natural stochastic models inspired by
juggling that got the attention of a few other members of the
community.  Exact combinatorial formulas for the stationary
distribution of these models were found, and proved by a beautiful
argument involving an auxiliary model defined on a larger state space
whose elements may be viewed as set partitions, for which the
stationary distribution is the uniform distribution.

The simplest model considered by Warrington, where the number of balls
is conserved, has been generalized by Leskel\"a and Varpanen as the
so-called Juggler's Exclusion Process (JEP), where the balls can be
thrown arbitrarily high so that the state space is infinite
\cite{LeskelaVarpanen}. These authors showed that the JEP converges to
a (unique) stationary distribution under very mild assumptions, could
obtain an exact expression for this stationary distribution in the
particular case of a ``memoryless'' (geometric) height distribution,
and noted an intriguing phenomenon of ``ultrafast convergence'' in
this case. More recently, a $q$-deformation of Warrington's original
finite model was considered by Engstr\"om, Leskel\"a and Varpanen
\cite{EngstromLeskelaVarpanen}, who also obtained exact expressions
for the stationary distribution via an enriched chain formulated in
terms of rook placements (in bijection with some set partitions).

In this paper, we provide multivariate generalizations of all the models
introduced in \cite{Warrington}, namely \emph{juggling, add-drop
  juggling} and \emph{annihilation juggling}. Furthermore, in the case
of the juggling model with a conserved number of balls, we investigate
the limiting case where balls can be thrown arbitrarily high, which
corresponds to the so-called infinite juggling model suggested by
Warrington. We also consider the limiting case where the number of
balls tends to infinity. In all these cases, we obtain an exact
formula for the stationary distribution (in particular, in the
case where balls can be thrown arbitrarily high, 
we find an ``exactly solvable'' instance of the general
JEP of \cite{LeskelaVarpanen} with countably many parameters) and the normalization
factor. In one case the normalization factor can be explicitly written as a 
homogeneous symmetric polynomial. 
Our proofs were obtained using two approaches. The direct approach involves 
guessing the general formulas (for instance by a computation for
small system sizes) and then proving the results
in a straightforward manner by considering the juggling process itself.
The other approach is more combinatorial, and consists in introducing an
enriched chain whose stationary distribution is simpler, and which
yields the original chain by a projection or ``lumping'' procedure
on set partitions or words. 

We remark that juggling models (especially the add-drop and annihilation versions) have a natural interpretation in statistical physics, where one could think of the balls as coming from a reservoir of particles. A model in this vein has been proposed recently, see \cite{AritaBouttierKrapivskyMallick}.

The rest of the paper is organized as follows. In
Section~\ref{sec:juggle}, we concentrate on the simplest version of
our model, which we call the Multivariate Juggling Markov Chain
(MJMC).  The model is defined in Section~\ref{subsec:juggle-def}, and
we also discuss the uniqueness of the stationary distribution (a
technicality deemed obvious in \cite{Warrington}).  The expression for
the stationary distribution of the
Markov chain is stated in Section~\ref{subsec:prob-juggle}. For
pedagogical purposes, we decide to provide two independent proofs by
the direct and the combinatorial approaches. The direct proof is given
in Section~\ref{subsec:mjmcpar}, via an interesting
reformulation of the MJMC in terms of integer partitions restricted to
lie within a rectangle. The
combinatorial proof comes in Section~\ref{subsec:enriched} and
involves set partitions with a prescribed number of elements and
blocks.
We then turn to extended models. Extensions to infinite state spaces
are considered in Section~\ref{sec:infextensions}: the case of a
finite number of balls but \emph{unbounded} heights (UMJMC) is
discussed in Section~\ref{subsec:unb}, while the case of an
\emph{infinite} number of balls (IMJMC) is considered in
Section~\ref{subsec:inf}. In both cases, we find the stationary
measure by the direct approach. Extensions to a fluctuating number of
balls (but with a finite state space) are considered in
Section~\ref{sec:flucext}: we provide the multivariate extension of
the add-drop and the annihilation models introduced in
\cite{Warrington}, in the respective
Sections~\ref{subsec:juggle-adddrop} and
\ref{subsec:juggle-annihilation}. These models have the same
transition graph, only the transitions probabilities differ.
In both cases we find the
stationary distribution by the combinatorial approach. In the case of
the annihilation model, we further observe the
interesting property that the stationary distribution is attained in
bounded time.
We end with some remarks and questions for future study in Section~\ref{sec:conclusion}.

The claims of this paper can be verified by downloading the Maple${}^{\text{TM}}$ program \texttt{RandomJuggling} either from the \texttt{arXiv} source or the first author's (A.A.) homepage.

\section{The finite Multivariate Juggling Markov Chain}
\label{sec:juggle}

\subsection{Definition}  
\label{subsec:juggle-def}

In this section, we introduce our juggling model in the simplest
setting, i.e.~a Markov chain on a finite state space.  We start by
explaining the model in colloquial terms, and refer to
\cite{Warrington} for further motivation.  Consider a person, called
\emph{Magnus} with no loss of generality, who is juggling with a fixed
finite number $\ell$ of balls. Time is discretized in, say, steps of
one second and we assume that, at each second, Magnus is able to catch
at most one ball, and then throws it back immediately. Besides this
limitation Magnus juggles \emph{perfectly}, i.e.~in such a way that the
ball will always return to him after some time (an imperfect juggler
could drop a ball or throw it in a wrong direction, for
instance). Magnus controls the velocity at which he sends the ball,
which determines how long it will take for the ball to return to him.
We suppose for now that the launch velocity is bounded or, in other
words, that the number of seconds before the ball returns to Magnus is
bounded by an integer $h$.

Ignoring further spatial constraints, a simplified description of the
state of the balls at a given time consists in associating to each
ball the (integer) number of seconds remaining before it is caught by
Magnus. This is now known as the \emph{siteswap} notation.
Of course, to be able to juggle for an indefinite amount of
time, Magnus must choose the successive launch velocities in such a
way that no two balls arrive to him at the same time. Thus the numbers
associated to different balls shall be distinct and, treating the
balls as indistinguishable, there are $\binom{h}{\ell}$ possible ball
states. It is convenient to think of a state as a configuration of
$\ell$ non overlapping particles on a one-dimensional lattice with $h$
sites, where the $i$-th site (read from the left) is occupied if and
only if a ball is scheduled to arrive $i$ seconds in the
future. However, beware that sites do no correspond to actual spatial
positions, but to the ``remaining flight times'' of the balls.
We denote by $k=h-\ell$ the number of empty (unoccupied) sites.

In this language, the time evolution of a state is easy to describe:
at each time step, all particles are moved one site to the left. If
there is no particle on the first site (i.e.~Magnus catches no ball),
then nothing else has to be done. Otherwise the particle on the first
site, which would exit the lattice if moved to the left, is instead taken
away and \emph{reinserted} at one of the $k+1$ available (empty) sites
on the lattice (determined by the launch velocity chosen by Magnus).
This defines the transition graph of our model, illustrated on
Figure~\ref{fig:example_markov_juggling} for $h=4$ and $\ell=k=2$
(ignoring edge labels for now).

We now assume that Magnus juggles at random: each reinsertion is made
randomly at one of the $k+1$ available sites. In our model, we assume
that the reinsertion is made at the $i$-th available site (read from
the left) with probability $x_{i-1}$, independently of the past, so
that our model is a Markov chain. Here, $x_0,\ldots,x_k$ are fixed
nonnegative real numbers such that $x_0+\cdots+x_k=1$. This defines
the \emph{Multivariate Juggling Markov Chain} (MJMC), which
generalizes the model considered in \cite{Warrington}, obtained by
taking $x_0=\cdots=x_k=1/(k+1)$, but is a particular case of the
general ``Juggling Exclusion Process'' defined in
\cite{LeskelaVarpanen}: beyond the extension to infinitely many empty
sites which we will discuss in the next section, the main difference
is that, in the model of Leskel\"a and Varpanen, the $x_i$'s are allowed to
depend on the current state.

\begin{figure}[h]
\begin{center}
\begin{tikzpicture} [>=triangle 45]
\draw (3.0,-1) node[circle,inner sep=2pt,draw] {};
\draw (3.3,-1) node[circle,inner sep=2pt,draw] {};
\draw (3.6,-1) node[circle,inner sep=2pt,fill=black,draw] {};
\draw (3.9,-1) node[circle,inner sep=2pt,fill=black,draw] {};

\draw (3.0,1) node[circle,inner sep=2pt,fill=black,draw] {};
\draw (3.3,1) node[circle,inner sep=2pt,draw] {};
\draw (3.6,1) node[circle,inner sep=2pt,draw] {};
\draw (3.9,1) node[circle,inner sep=2pt,fill=black,draw] {};

\draw (0.0,3) node[circle,inner sep=2pt,fill=black,draw] {};
\draw (0.3,3) node[circle,inner sep=2pt,fill=black,draw] {};
\draw (0.6,3) node[circle,inner sep=2pt,draw] {};
\draw (0.9,3) node[circle,inner sep=2pt,draw] {};

\draw (6.0,3) node[circle,inner sep=2pt,draw] {};
\draw (6.3,3) node[circle,inner sep=2pt,fill=black,draw] {};
\draw (6.6,3) node[circle,inner sep=2pt,fill=black,draw] {};
\draw (6.9,3) node[circle,inner sep=2pt,draw] {};

\draw (3.0,5) node[circle,inner sep=2pt,fill=black,draw] {};
\draw (3.3,5) node[circle,inner sep=2pt,draw] {};
\draw (3.6,5) node[circle,inner sep=2pt,fill=black,draw] {};
\draw (3.9,5) node[circle,inner sep=2pt,draw] {};

\draw (3.0,7) node[circle,inner sep=2pt,draw] {};
\draw (3.3,7) node[circle,inner sep=2pt,fill=black,draw] {};
\draw (3.6,7) node[circle,inner sep=2pt,draw] {};
\draw (3.9,7) node[circle,inner sep=2pt,fill=black,draw] {};

\draw [->] (2.9,4.7) -- node [right]  {$x_0$} (1.1,3.0) ;
\draw [->] (4.1,4.7) -- node [right] {$x_1$} (5.7,3.3);
\draw [->] (3.3,5.3) -- node [left] {$x_2$} (3.3,6.7);

\draw [->] (4.1,-0.7) -- node [right] {$1$} (5.9,2.7); 

\draw [->] (3.6,6.7) -- node [right] {$1$} (3.6,5.3);

\draw [->,out=-90,in=-90,looseness=2.5] (6.4,2.7) to node [below] {$1$} (0.4,2.7);

\draw [->] (3.3,0.7) -- node [left] {$x_2$} (3.3,-0.7);
\draw [->] (4.1,1.3) -- node [left] {$x_1$} (5.7,2.7); 
\draw [->] (3.5,1.3) -- node [left] {$x_0$} (3.5,4.7); 

\draw [->,out=135,in=45,looseness=10] (0.3,3.3) to   node[above]  {$x_0$}(0.6,3.3);
\draw [->] (0.9,3.3) -- node [above]  {$x_1$} (2.7,5) ;
\draw [->] (1.1,2.7) -- node [right] {$x_2$} (2.7,1.3);
\end{tikzpicture}
\vspace{-0.5cm}
\caption{The Markov chain with $h=4$ and $\ell=k=2$.}
\label{fig:example_markov_juggling}
\end{center}
\end{figure}
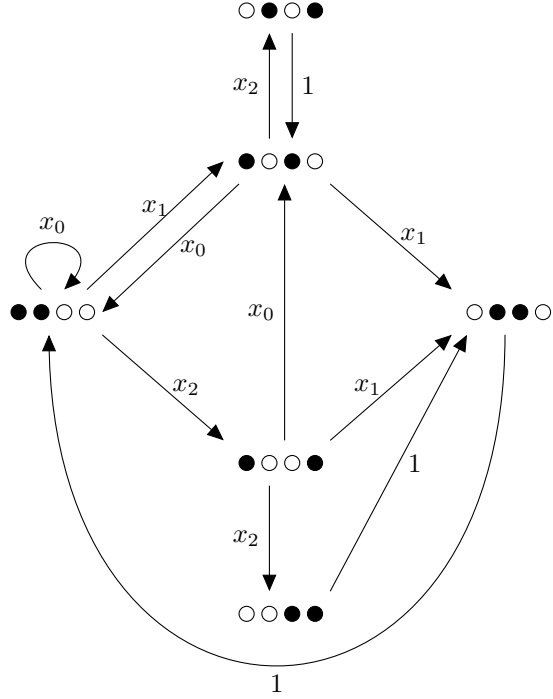

We now provide a more formal mathematical definition of the MJMC.
Following Warrington's notation, let $\mathrm{St}_h$ denote the set
of words of length $h$ on the alphabet $\{\bullet,\circ\}$, and let
$\mathrm{St}_{h,k} \subset \mathrm{St}_h$ be the subset of words
containing exactly $k$ occurrences of $\circ$.
For $A \in \mathrm{St}_{h,k+1}$ and $i \in \{0,\ldots,k\}$, we let
$T_i(A) \in \mathrm{St}_{h,k}$ be the word obtained by replacing the
$(i+1)$-th occurrence of $\circ$ in $A$ by $\bullet$.

\begin{Def} \label{Def:MJMC}
  Given $h,k$ nonnegative integers such that $h \geq k$, and
  $x_0,\ldots,x_k$ nonnegative real numbers such that
  $x_0+\cdots+x_k=1$, the \emph{Multivariate Juggling Markov Chain} is
  the Markov chain on the state space $\mathrm{St}_{h,k}$ for which
  the transition probability from $A=a_1 a_2 \cdots a_{h}$ to $B$ reads
  \begin{equation}
    P_{A,B}=
    \begin{cases}
      1 & \text{if $a_1=\circ$ and $B=a_2 \cdots a_h \circ$,} \\
      x_i & \text{if $a_1=\bullet$ and $B=T_i(a_2 \cdots a_h \circ)$,} \\
      0 & \text{otherwise.}
    \end{cases}
    \label{eq:jugp}
  \end{equation}
\end{Def}

\begin{exmp}
  Figure \ref{fig:example_markov_juggling} illustrates the Markov
  chain in the case $h=4$, $k=2$, and the transition matrix in the
  basis
  \(
  (\bullet \bullet \circ \circ, \bullet \circ \bullet \circ,
  \bullet \circ \circ \bullet , \circ \bullet \bullet \circ, \circ
  \bullet \circ \bullet , \circ \circ \bullet \bullet)
  \) 
  reads
  \begin{equation}
    \begin{pmatrix}
      x_0 & x_1 & x_2 & 0 & 0 & 0\\
      x_0 & 0 & 0 & x_1 & x_2 & 0\\
      0 & x_0 & 0 & x_1 & 0 & x_2\\
      1 & 0 & 0 & 0 & 0 & 0\\
      0 & 1 & 0 & 0 & 0 & 0\\
      0 & 0 & 0 & 1 & 0 & 0
    \end{pmatrix}
  \end{equation}
  Note that $(1,x_1+x_2,x_2,(x_1+x_2)^2,x_2(x_1+x_2),x_2^2)$ is a
  left eigenvector with eigenvalue $1$.
\end{exmp}

To compare our model with that of \cite{LeskelaVarpanen}, note that
$\mathrm{St}_{h,k}$ may be canonically identified with the set of
subsets of $\{1,\ldots,h\}$ with $\ell=h-k$ elements (corresponding to
the positions of the occupied sites). For two such subsets $S$ and
$S'$, the transition probability from $S$ to $S'$ reads
\begin{equation}
  \label{eq:MJMCsubset}
  P_{S,S'} =
  \begin{cases}
    1 & \text{if $S'=S-1$,}\\
    x_i & \text{if $S'=(S-1) \setminus \{0\} \cup \{j\}$ with $i=|\{1,\ldots,j\} \setminus S|$,}\\
    0 & \text{otherwise,}
  \end{cases}
\end{equation}
where $S-1$ is the set whose entries are those of $S$ with 1
subtracted.

\begin{prop} \label{prop:finirred} If $x_0>0$, then the MJMC has a
  unique closed communicating class, whose all states are
  aperiodic. If furthermore $x_k>0$, then the MJMC is irreducible.
\end{prop}

\begin{proof}
  Assume that $x_0>0$ and let $\mathcal{E}=\bullet^\ell\circ^k$ be the
  ``lowest'' state. It is clear that, starting from any initial state,
  we may obtain $\mathcal{E}$ by a sequence of transitions of
  probabilities $x_0$ or $1$, simply by doing any reinsertion at the
  first available site. Since $\mathcal{E}$ jumps to itself with
  probability $x_0$, there is a unique closed communicating class, and
  all states in this class are aperiodic.

  Assume now that $x_k>0$: to prove irreducibility we simply need to
  show that any state $A=a_1 \cdots a_h \in \mathrm{St}_{h,k}$ may be
  obtained from $\mathcal{E}$. The idea is to build the prefix of $A$
  of length $i$ at the $i$-th step so that we get $A$ in exactly $h$ steps.
  For $i \in \{0,\ldots,h\}$, we let
  $A_i=\bullet^{\ell-i+n_i} \circ^{k-n_i} a_1 \cdots a_i$ where $n_i$
  is the number of occurrences of $\circ$ in $a_1\cdots a_i$, so that
  $A_i \in \mathrm{St}_{h,k}$, with $A_0=\mathcal{E}$ and $A_h=A$.
  For $i<h$, we have
  \begin{equation}
    P_{A_i,A_{i+1}} =
    \begin{cases}
      1 & \text{if $\ell=i-n_i$},\\
      x_0 & \text{if $\ell>i-n_i$ and $a_{i+1}=\circ$},\\
      x_k & \text{if $\ell>i-n_i$ and $a_{i+1}=\bullet$}.\\
    \end{cases}
  \end{equation}
  Note that when $a_1 = \circ$, $A_1 = \mathcal{E}$, but this is allowed
  since there is always a self-loop from $\mathcal{E}$ to itself.
  Thus $A$ may be obtained from $\mathcal{E}$ by a sequence of
  possible transitions.
\end{proof}

\begin{rem}
  When $x_0=0$, the chain may have several closed communicating
  classes. For example, see Figure~\ref{fig:example_markov_juggling} and let 	$x_0=x_1=0$.
\end{rem}

\subsection{Stationary distribution}
\label{subsec:prob-juggle}

From now on we assume $x_0>0$. By Proposition~\ref{prop:finirred}, the
MJMC admits a unique stationary probability distribution. Our main
result is an explicit form for it, given as follows.

\begin{thm} \label{thm:prob-juggle}
  The stationary distribution $\pi$ of the MJMC is given by
  \begin{equation}
    \label{eq:prob-juggle}
    \pi(B) = \frac{1}{Z_{h,k}} \prod_{\substack{i\in \{1,\ldots,h\} \\ b_i = \bullet}}
    \left( x_{E_i(B)} + \cdots + x_k \right),
  \end{equation}
  where $B=b_1 \cdots b_h \in \mathrm{St}_{h,k}$ and $E_i(B) = \#
  \{j<i | b_j = \circ \}$, and where $Z_{h,k} \equiv
  Z_{h,k}(x_0,\ldots,x_k)$ is determined by the condition that $\pi$
  be a probability distribution.
\end{thm}

We will provide two proofs of this theorem, as they both lead to
interesting generalizations to be studied later in Sections
\ref{sec:infextensions} and \ref{sec:flucext}. The first proof,
presented in Section~\ref{subsec:mjmcpar}, is a rather direct
computational check that \eqref{eq:prob-juggle} indeed defines a
stationary measure of the MJMC (but we find it clearer to first
reformulate the chain in terms of integer partitions). The second proof,
presented in Section~\ref{subsec:enriched}, is more combinatorial and
generalizes the approach of \cite{Warrington}: it consists in defining
an ``enriched'' Markov chain on a larger state space, where the
stationary probability of a given state will be given by a monomial in
$x_0,\ldots,x_k$. The MJMC is then recovered by lumping together
states of the enriched chain, which explains the form of
\eqref{eq:prob-juggle}.

\noindent
Before proceeding to the proofs, let us discuss a few interesting cases.

\begin{exmp} \label{exmp:unif}
  If the reinsertion is made uniformly among all $k+1$ available
  sites, i.e.\ if we take $x_0=\cdots=x_k=1/(k+1)$, then we recover
  \cite[Theorem 1]{Warrington} from Theorem~\ref{thm:prob-juggle} (see
  Corollary~\ref{cor:Zexplic} below for a short rederivation of the
  normalization factor).
\end{exmp}

\begin{exmp} \label{exmp:elv}
  If the reinsertion site is chosen according to a geometric
  distribution of parameter $q$ conditioned to be smaller than or
  equal to $k+1$, i.e.\ if we take $x_i=(1-q)q^i/(1-q^{k+1})$,
  $i=0,\ldots,k$, then we recover the so-called bounded geometric JEP
  considered in \cite{EngstromLeskelaVarpanen}. It is a simple exercise
  to check that we recover \cite[Theorem 2.1]{EngstromLeskelaVarpanen}
  from Theorem~\ref{thm:prob-juggle} (see also
  Corollary~\ref{cor:Zexplic} below for the normalization factor).
\end{exmp}

\begin{exmp} \label{exmp:truncgeom}
  A slight variant of the previous example consists in picking the
  reinsertion site according to a ``truncated'' geometric distribution
  of parameter $q$, i.e.\ taking $x_i=(1-q)q^i$ for $i=0,\ldots,k-1$
  and $x_k=q^k$. In that case, the stationary probability of $B \in
  \mathrm{St}_{h,k}$ reads
  \begin{equation}
    \label{eq:truncgeom}
    \pi(B) = \frac{1}{\binom{h}{k}_q}
    \prod_{\substack{i\in \{1,\ldots,h\} \\ b_i = \bullet}} q^{E_i(B)}
  \end{equation}
  (see again Corollary~\ref{cor:Zexplic} below for the
  normalization factor).
\end{exmp}

We now move on to discuss properties of the normalization factor $Z_{h,k}$.
Observe first that the expression
\begin{equation}
  Z_{h,k}(x_0,\ldots,x_k)= \sum_{B \in \mathrm{St}_{h,k}} \prod_{\substack{i\in \{1,\ldots,h\} \\ b_i = \bullet}}
  \left( x_{E_i(B)} + \cdots + x_k \right)\label{eq:Zworddef}
\end{equation}
allows us to define it for any $x_0,\ldots,x_k$ (not necessarily positive
or subject to the condition $x_0+\cdots+x_k=1$) and $h,k \geq 0$. It
vanishes for $h<k$ and is otherwise a homogeneous polynomial of
degree $\ell=h-k$. 

\begin{prop}
  \label{prop:Zhform}
  We have
  \begin{equation}
    \label{eq:Zhform}
    Z_{h,k} = \sum_{0 \leq i_1 \leq \cdots \leq i_\ell \leq k} y_{i_1} \cdots y_{i_\ell}
    = h_\ell(y_0,\ldots,y_k)
  \end{equation}
  with $h_\ell$ the complete homogeneous symmetric polynomial of
  degree $\ell=h-k$, and
  \begin{equation}
    \label{eq:ydef}
    y_m = \sum_{j=m}^k x_j, \qquad m=0,\ldots,k.
  \end{equation}
\end{prop}

Again, there are several ways to establish this result. Perhaps the
most combinatorial explanation comes from the integer partition
approach of Section~\ref{subsec:mjmcpar}, see Remark~\ref{rem:Zparform}
below. But, starting from the definition \eqref{eq:Zworddef}, we may
easily check \eqref{eq:Zhform} by induction, using the following:

\begin{lem} \label{lem:Zsetrec}
  For $h,k \geq 0$, we have
  \begin{equation}
    Z_{h,k}(x_0,\ldots,x_k) = 
    Z_{h-1,k-1}(x_1,\ldots,x_k) + (x_0+\cdots+x_k)
    Z_{h-1,k}(x_0,\ldots,x_k)\label{eq:Zsetrec}
  \end{equation}
  where, by convention, $Z_{h,k}=\delta_{h,k}$ if $h=-1$ or $k=-1$.
\end{lem}

\begin{proof}
  Obtained immediately by distinguishing whether $B \in
  \mathrm{St}_{h,k}$ starts with a $\circ$ or a $\bullet$ (for $h<k$
  all terms vanish as wanted).
\end{proof}

We readily verify by induction and homogeneity some explicit specializations:

\begin{cor} \label{cor:Zexplic}
  We have
  \begin{align}
    &Z_{h,k}(1,\ldots,1) = \stirling{h+1}{k+1}, \label{eq:Zunif}\\
    &Z_{h,k}(q^k,q^{k-1},\ldots,1) = \stirling{h+1}{k+1}_q, \label{eq:Zqanal} \\
    &Z_{h,k}(1,q,\ldots,q^k) = q^{k(h-k)} \stirling{h+1}{k+1}_{1/q}, \\
    &Z_{h,k}((1-q),(1-q)q,\ldots,(1-q)q^{k-1},q^k) = \binom{h}{k}_q,
  \end{align}
  where $\stirling{\cdot}{\cdot}$ denotes Stirling numbers of the
  second kind, $\stirling{\cdot}{\cdot}_q$ their $q$-analogues as
  defined in \cite{Gould}
  and $\binom{\cdot}{\cdot}_q$ $q$-binomial coefficients. Using
  homogeneity, we deduce the normalization factors for
  Examples~\ref{exmp:unif} and \ref{exmp:elv}.
\end{cor}

For completeness, we mention another recursion relation for $Z_{h,k}$
of a different nature, as we remove $x_k$ instead of $x_0$.

\begin{prop} \label{prop:zrecur}
  For $h,k \geq 0$, we have
  \begin{equation} \label{eq:zrecur}
    Z_{h,k}(x_0,\ldots,x_k) =
    \sum_{n=0}^{h-k} \binom{h}{n} x_k^{n} Z_{h-n-1,k-1}(x_0,\ldots,x_{k-1})
  \end{equation}
  where, by convention, $Z_{h,k}=\delta_{h,k}$ if $h=-1$ or $k=-1$.
\end{prop}

\begin{proof}
  Write
  \begin{equation}
    Z_{h,k}=h_\ell(y'_0+x_k,y'_1+x_k,\ldots,y'_{k-1}+x_k,x_k)
  \end{equation}
  with $y'_m=\sum_{j=m}^{k-1} x_j$, and use the identity (easily
  verified e.g.\ by induction on $\ell+k$)
  \begin{equation}
    h_\ell(a_0+a,a_1+a,\ldots,a_k+a) = \sum_{n=0}^\ell \binom{\ell+k}{n} a^n
    h_{\ell-n}(a_0,\ldots,a_k).
  \end{equation}
\end{proof}

\subsection{Reformulation of the MJMC in terms of integer partitions}
\label{subsec:mjmcpar}

\begin{figure}[htpb]
  \centering
  \includegraphics[width=1\textwidth]{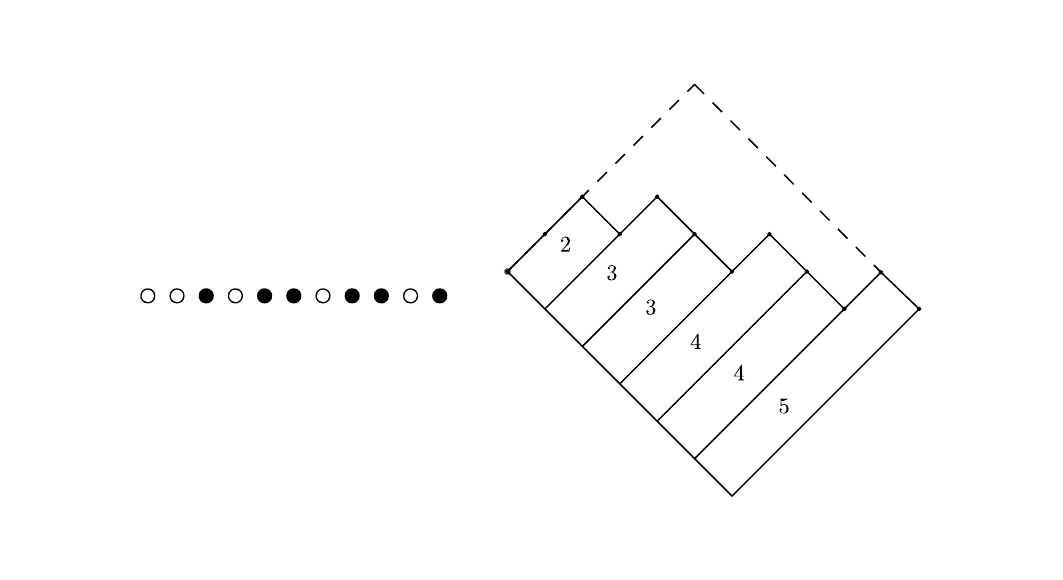}
  \caption{Bijection between juggling states and integer
    partitions. Each $\circ$ is replaced by a north-east step, and
    each $\bullet$ by a south-east step. The partition thus obtained
    is $(5,4,4,3,3,2)$.}
  \label{fig:bijection_states_partitions}
\end{figure}

There is a natural bijection between $\mathrm{St}_{h,k}$ and
$\mathrm{Par}_{k,h-k}$, where $\mathrm{Par}_{k,\ell}$ is the set of
integer partitions whose Young diagram fits within a $k \times \ell$
rectangle. In other words, an element of $\mathrm{Par}_{k,\ell}$ can be
viewed as a nonincreasing sequence of $\ell$ nonnegative integers
smaller than or equal to $k$.  The bijection is given explicitly as follows: given
a state in $\mathrm{St}_{h,k}$, we denote by $s_1 < \cdots < s_\ell$ 
the positions of $\bullet$'s, then the corresponding integer partition is
$(s_\ell-\ell,s_{\ell-1}-(\ell-1),\ldots,s_1-1)$.
Alternatively, an equivalent graphical construction is displayed on
Figure~\ref{fig:bijection_states_partitions}. Upon identifying
$\mathrm{St}_{h,k}$ and $\mathrm{Par}_{k,h-k}$ via this bijection, we
may reformulate the MJMC as follows.

\begin{prop}
  For $\lambda = (\lambda_1,\ldots,\lambda_{\ell})$ and $\mu$
  two partitions in
  $\mathrm{Par}_{k,\ell}$, the transition probability from $\lambda$
  to $\mu$ reads
  \begin{equation}
    P_{\lambda,\mu}=
    \begin{cases}
      1 & \text{if $\lambda_{\ell}\neq 0$ and
        $\mu=(\lambda_1-1,\ldots,\lambda_{\ell}-1)$,} \\
      x_i & \text{if $\lambda_{\ell}=0$ and there exists $j \in \{1,\ldots,\ell\}$ such that} \\
      & \text{$\mu=(\lambda_1-1, \ldots ,\lambda_{j-1}-1,i,\lambda_{j},\ldots,\lambda_{\ell-1})$,}\\
      0 & \text{otherwise.}
    \end{cases}
    \label{eq:part}
\end{equation}
\end{prop}

The proof of this proposition is left to the reader. The interest of
the reformulation in terms of integer partitions is that the
stationary distribution of the MJMC takes a particularly simple form:
indeed, Theorem~\ref{thm:prob-juggle} is equivalent to

\begin{thm}
  \label{thm:finstat}
  The stationary distribution $\pi$ of the MJMC is given by
  \begin{equation}
    \pi(\lambda) = \frac{1}{Z_{h,k}} \prod_{i=1}^\ell y_{\lambda_i}
  \end{equation}
  where $\lambda=(\lambda_1,\ldots,\lambda_\ell)$ is an element of
  $\mathrm{Par}_{k,\ell}$ and where the $y_m$ are as in \eqref{eq:ydef}.
\end{thm}

\begin{proof}
  Let $w(\lambda)=\prod_{i=1}^\ell y_{\lambda_i}$ be the unnormalized
  weight of $\lambda \in \mathrm{Par}_{k,\ell}$.  We simply check that
  \begin{equation}
    w(\lambda) = \sum_{\mu \in \mathrm{Par}_{k,\ell}} w(\mu) P_{\mu,\lambda}
  \end{equation}
  by considering the possible predecessors of $\lambda$:
  \begin{itemize}
  \item If $\lambda_1 = k$ then the only possible predecessor for
    $\lambda$ is $\mu =
    (\lambda_2,\ldots,\lambda_{\ell},0)$, so that
    $w(\mu)=\prod_{i=2}^{\ell} y_{\lambda_i}$ (because $y_0 = 1$) and
    $P_{\mu,\lambda} = x_k = y_k$. Thus we have $w(\lambda) =
    P_{\mu,\lambda}w(\mu)$ as wanted.
  \item We now assume that $\lambda_1 < k$. Let $\mu$ be a predecessor
    of $\lambda$, $\mu$ can 
  \begin{itemize}
  \item either be of the form
    $\mu=(\lambda_1+1,\ldots,\lambda_{\ell}+1)$, i.e.\ is associated
    with a word that starts with a $\circ$, in which case we have
    $w(\mu) = \prod_{i=1}^{\ell} y_{\lambda_i+1} = \prod_{i=1}^\ell
    (y_{\lambda_i}-x_{\lambda_i})$ and $P_{\mu,\lambda} = 1$;
  \item or be of the form
    $\mu=(\lambda_1+1,\ldots,\lambda_{j-1}+1,\lambda_{j+1},\ldots,\lambda_\ell,0)$
    for any $1\leq j\leq \ell$, i.e.\ is associated with a word that
    starts with a $\bullet$, in which case we have $w(\mu) =
    \prod_{i=1}^{j-1} (y_{\lambda_i}-x_{\lambda_i})
    \prod_{i=j+1}^{\ell} y_{\lambda_i}$ and $P_{\mu,\lambda} =
    x_{\lambda_j}$.
  \end{itemize}
  Thus, we have
  \begin{equation} \label{eq:finstatverif}
    \begin{split}
      \sum_{\mu \in \mathrm{Par}_{k,\ell}} w(\mu) P_{\mu,\lambda} &=
      \prod_{i=1}^\ell (y_{\lambda_i}-x_{\lambda_i}) +
      \sum_{j=1}^{\ell} \left(x_{\lambda_j}
        \prod_{i=1}^{j-1}(y_{\lambda_i}-x_{\lambda_i})\prod_{i=j+1}^{\ell}
        y_{\lambda_i}\right) \\
      &= \prod_{i=1}^\ell y_{\lambda_i} = w(\lambda)
    \end{split}
  \end{equation}
  where, to go from the first line to the second, we write
  $x_{\lambda_j} = y_{\lambda_j} - (y_{\lambda_j}-x_{\lambda_j})$
  so that the sum on the right hand side  becomes telescopic.
\end{itemize}\end{proof}

\begin{exmp} \label{exmp:parpar}
  We may now revisit Examples \ref{exmp:unif}, \ref{exmp:elv} and
  \ref{exmp:truncgeom} in the language of integer partitions. The
  corresponding stationary probabilities of $\lambda \in
  \mathrm{Par}_{k,\ell}$ read
  \begin{equation}
    \label{eq:parexmp}
    \pi(\lambda) =
    \begin{cases}
      \frac{1}{\stirling{h+1}{k+1}} \prod_{i=1}^\ell (k+1-\lambda_i) &
      \text{for $x_i=\frac{1}{k+1}$,} \\
      \frac{1}{q^{k \ell} \stirling{h+1}{k+1}_{1/q}} \prod_{i=1}^\ell q^{\lambda_i} [k+1-\lambda_i]_q &
      \text{for $x_i=\frac{q^i}{[k+1]_q}$,} \\
      \frac{1}{\binom{h}{k}_q} \prod_{i=1}^\ell q^{\lambda_i} &
      \text{for $x_i=(1-q)^{1-\delta_{i,k}} q^i$.} \\
    \end{cases}
  \end{equation}
\end{exmp}

\begin{rem}
  \label{rem:Zparform}
  One can immediately recover Proposition~\ref{prop:Zhform} from
  Theorem~\ref{thm:finstat}, as
  \begin{equation}
    \label{eq:Zparform}
    Z_{h,k} = \sum_{\lambda \in \mathrm{Par}_{k,\ell}}
    \prod_{j=1}^{\ell} y_{\lambda_j} =
    \sum_{\substack{(m_0,\ldots,m_k) \in \mathbb{N}^{k+1} \\
        \sum_{i=0}^k m_i=\ell}} \prod_{i=0}^k y_i^{m_i} = h_\ell(y_0,\ldots,y_k)
  \end{equation}
  where the $m_i$ correspond to the part multiplicities.
  A slight extension of this reasoning consists in restricting the sum
  to partitions $\lambda$ such that, say, $\lambda_j=n$ (with $1\leq j
  \leq \ell$ and $0\leq n\leq k$), so as to obtain the probability that the
  $j$-th part in the stationary distribution $\pi$ is $n$:
  \begin{equation}
    \pi(\{\lambda_j=n\}) = \frac{y_n h_{j-1}(y_n,\ldots,y_k)
      h_{\ell-j}(y_0,\ldots,y_n)}{h_{\ell}(y_0,\ldots,y_k)}.
  \end{equation}
  More generally, the joint distribution for a fixed number of parts reads
  \begin{equation}
    \pi(\{\lambda_{j_1}=n_1,\ldots,\lambda_{j_m}=n_m\}) =
    \frac{\prod_{s=1}^{m+1} y_{n_s} h_{j_s-j_{s-1}-1}(y_{n_s},\ldots,y_{n_{s-1}})}{
      h_{\ell}(y_0,y_1,\ldots,y_k)}
  \end{equation}
  where $0=j_0 < \cdots < j_{m+1}=\ell+1$ and $k =n_0 
 \geq \cdots \geq n_{m+1} = 0$ (with $y_0=1$).
\end{rem}

\subsection{Enriched Markov chain on set partitions}
\label{subsec:enriched}

We now provide another, more combinatorial, proof of
Theorem~\ref{thm:prob-juggle} whose rough idea goes as follows:
consider the stationary probability \eqref{eq:prob-juggle} and expand
the product in the right hand side as a sum of monomials in the $x_i$'s. We will
interpret each of these monomials as the stationary probability of an
``enriched'' state belonging to the larger state space of another
Markov chain. The MJMC will then be obtained as a projection of this
enriched chain (see e.g.\ \cite[Section 2.3.1]{LevinPeresWilmer} for a
definition of this notion).

\subsubsection{Definitions and basic properties}
\label{sec:enricheddef}

Our construction is a generalization of that of Warrington's.
Rather than working with ``landing/throwing-states'', we
prefer to work directly in the language of set partitions, see
\cite[Lemma 2]{Warrington}.  Recalling that $h$ stands for the total
number of sites and $k$ for the number of empty sites in the particle
picture of the MJMC, let us introduce the shorthand notations
\begin{equation}
  H = h+1 \qquad \text{and} \qquad K=k+1.
\end{equation}
An \emph{enriched state} will then be a partition of the set
$\{1,\ldots,H\}$ into $K$ subsets or \emph{blocks}.  We denote
by $\mathcal{S}(H,K)$ the set of enriched states, and recall that
$\stirling{H}{K}=|\mathcal{S}(H,K)|$ is a Stirling number of the second
kind. To each enriched state $\sigma$, we associate a word
$\psi(\sigma)=a_1 \ldots a_h$ by setting, for all $i$ between $1$ and
$h$, $a_i=\circ$ if $i$ is a \emph{block maximum} of $\sigma$
(i.e. the largest element of its block in $\sigma$), and $a_i=\bullet$
otherwise. Observe that $\psi$ is a surjection from $\mathcal{S}(H,K)$
onto $\mathrm{St}_{h,k}$.

We now define the enriched Markov chain on $\mathcal{S}(H,K)$, which
requires some notations.
For an enriched state $\sigma$, we denote by $\sigma^\downarrow$ the
partition of the set $\{1,\ldots,h\}$ obtained by removing $1$ from
$\sigma$ (i.e. removing $1$ from its block in $\sigma$, and removing
this block from $\sigma$ if it becomes empty), and shifting all the
remaining elements of all blocks down by $1$. Note that $\sigma
\mapsto \sigma^\downarrow$ is a mapping from $\mathcal{S}(H,K)$ to
$\mathcal{S}(h,k) \cup \mathcal{S}(h,K)$ (which is 
surjective). For $\tau \in \mathcal{S}(h,K)$ and $i \in
\{0,\ldots,k\}$, we denote by $I_i(\tau)$ the set partition of
$\{1,\ldots,H\}$ obtained by inserting $H$ into the $(i+1)$-th block
of $\tau$, where the blocks are numbered by ascending order of their
maxima (i.e. the first block has the smallest maximum among all
blocks, etc. -- this differs from the so-called standard form which
consists in ordering blocks by ascending order of their minima). Note
that $I_i(\tau) \in \mathcal{S}(H,K)$ and that the mapping $(\tau,i)
\mapsto I_i(\tau)$ is injective.

\begin{Def}
  The \emph{enriched chain} is the Markov chain on $\mathcal{S}(H,K)$
  for which the transition probability from $\sigma$ to $\tau$ is
  given by
  \begin{equation} \label{eq:jugp-enriched}
    \tilde{P}_{\sigma,\tau} =
    \begin{cases}
      1 & \text{if $\{1\} \in \sigma$ and $\tau = \sigma^\downarrow \cup \{H\}$,} \\
      x_i & \text{if $\{1\} \notin \sigma$ and $\tau =
        I_i(\sigma^\downarrow)$
        for some $i \in \{0,\ldots,k\}$,}\\
        0 & \text{otherwise}.
    \end{cases}
  \end{equation}
\end{Def}

The condition $x_0+\cdots+x_k=1$ and the above remarks ensure that
$\tilde{P}$ is indeed a right stochastic matrix (i.e.\ each of its
rows sums to $1$).

\begin{exmp}
  For $H=8$ and $K=3$:
  \begin{itemize}
  \item the enriched state $1\,|\,3,5,6\,|\,2,4,7,8$ jumps with probability
	$1$ to $2,4,5\,|\,1,3,6,7\,|\,8$,
  \item the enriched state $\sigma = 3,5\,|\,2,6,7\,|\,1,4,8$
  	reaches the intermediate state 
  	$\sigma^\downarrow = 2,4\,|\,1,5,6\,|\,3,7$ and
    jumps with probability:
    \begin{itemize}
    \item $x_0$ to $1,5,6\,|\,3,7\,|\,2,4,8$,
    \item $x_1$ to $2,4\,|\,3,7\,|\,1,5,6,8$,
    \item $x_2$ to $2,4\,|\,1,5,6\,|\,3,7,8$.
    \end{itemize}
  \end{itemize}
  Note that here we write blocks in ascending order of their maxima,
  which differs from the standard notation of writing set partitions
  in ascending order of their minima.
\end{exmp}

\begin{figure}[thtp]
  \centering
  \includegraphics[width=\textwidth]{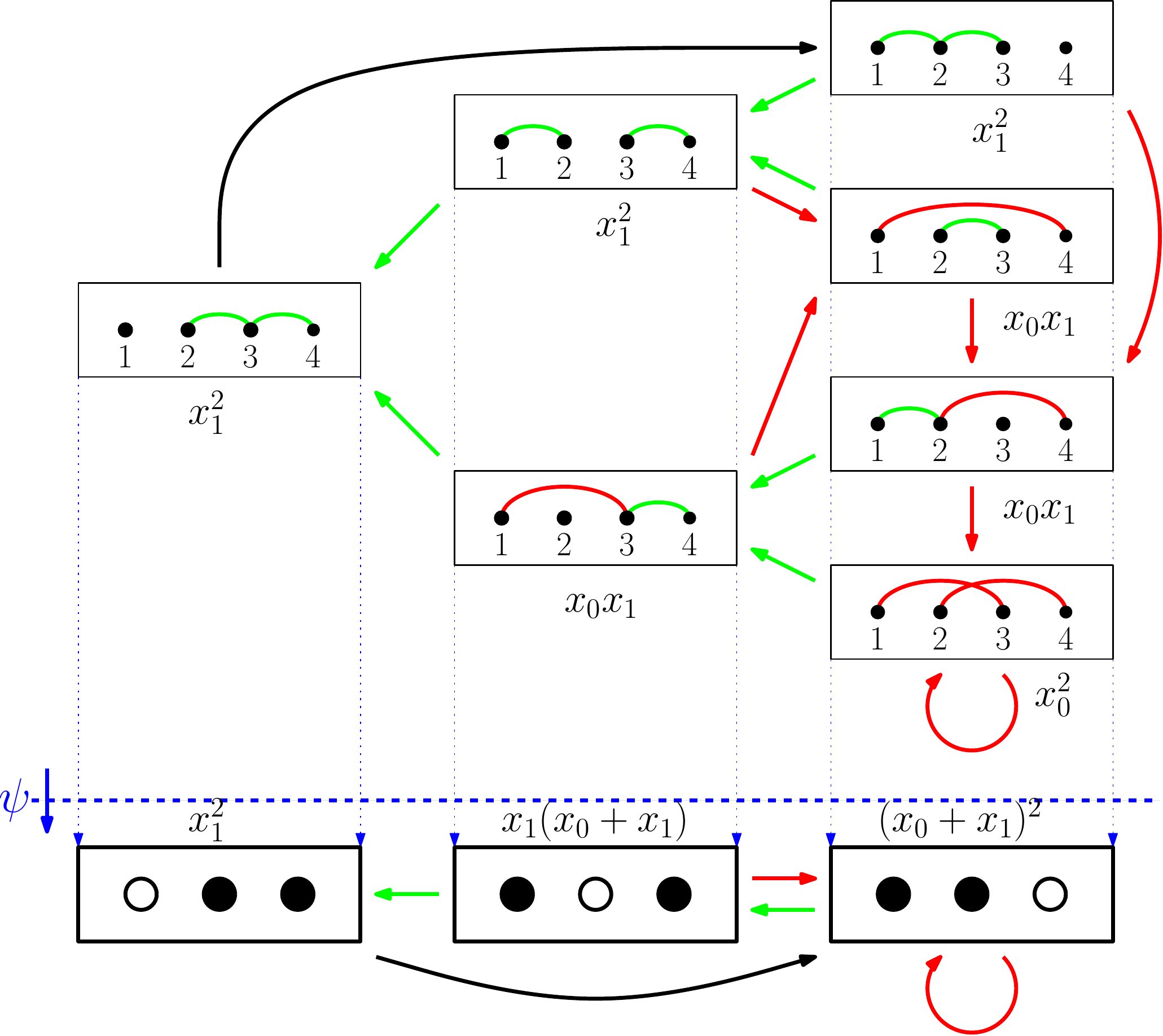}
  \caption{Transition graphs of the enriched chain (top part, above
    blue horizontal dashed line) and of the MJMC (bottom part) for
    $H=4$ and $K=2$ ($h=3$ and $k=1$). Enriched states (set
    partitions) are represented via ``arches'', see
    Definition~\ref{Def:arch}.  Red, green and black arrows represent
    transitions of respective probability $x_0$, $x_1$ and $1$. Next
    to each state is displayed its unnormalized stationary weight; for
    the enriched chain it is obtained by attaching a weight $x_1$
    (resp.\ $x_0$) to each arch covering 1 (resp.\ 2) blocks,
    displayed in green (resp.\ red). Enriched states in a same column
    map to the same MJMC state (displayed below) via $\psi$.}
  \label{fig:chain42}
\end{figure}

\begin{exmp}
  For $H=4$ and $K=2$, the transition graph of the enriched chain is
  illustrated on Figure~\ref{fig:chain42}.
\end{exmp}

The existence and uniqueness of the stationary distribution of the
enriched chain results from:

\begin{prop} \label{prop:enrichirred}
  For $x_0>0$, the enriched chain has a unique closed communicating
  class, whose all states are aperiodic. The chain is irreducible if
  and only if all $x_i$'s are nonzero.
\end{prop}

\begin{proof}
  The first assertion is proved in the same way as that of
  Proposition~\ref{prop:finirred}, the ``lowest'' state being now the
  set partition
  \begin{equation}
    \tilde{\mathcal{E}}=\left\{\{1,\ldots,H\} \cap (j+K \mathbb{N})
    \right\}_{j=1,\ldots,K}
    \label{eq:tildeEdef}
  \end{equation}
  (precisely, starting from any initial state, we eventually obtain
  $\tilde{\mathcal{E}}$ by applying only transitions of probabilities
  $x_0$ and $1$). Note that for $H=4$ and $K=2$, we have
  $\tilde{\mathcal{E}}=1,3\,|\,2,4$ (which is the state with two red
  arches in Figure~\ref{fig:chain42}).
  
  For the second assertion, suppose that $x_i>0$ for all $i$: we want
  to prove that the transition graph is strongly connected. We first
  replace the transitions of probability 1 by $K$ transitions of
  probabilities $x_0,\ldots,x_k$, so that each state has $K$ outgoing
  transitions, counting transitions to itself.  It is then not
  difficult to check from the definition of the enriched chain that
  each state also has $K$ incoming transitions. It follows that for
  any set of states, the number of incoming and outgoing transitions
  are the same. Since the state space is finite, all communicating
  classes are closed. As there is a unique closed communicating class,
  irreducibility follows. Conversely, if $x_i=0$ for some $i$, then
  clearly some set partitions, e.g.\ $1\,|\,2|\cdots|i|i+2|\cdots|K-1|K\,
  (K+1) \cdots h|(i+1)\, H$, are never produced.
\end{proof}

\subsubsection{Projection, stationary distribution and second proof of
  Theorem~\ref{thm:prob-juggle}}
\label{sec:enrichedintwn}

From the surjection $\psi$ introduced above, we define a rectangular
matrix $\Psi$ with rows indexed by elements of $\mathcal{S}(H,K)$ and
columns indexed by elements of $\mathrm{St}_{h,k}$ in the obvious
manner (namely a coefficient of $\Psi$ is $1$ if its column index is
the image by $\psi$ of its row index, and $0$ otherwise).

\begin{lem} \label{lem:intertwin}
  We have the intertwining relation
  \begin{equation} \label{eq:intertwin}
    \tilde{P} \Psi = \Psi P.
  \end{equation}
\end{lem}

For clarity, the proofs of all lemmas in this section are postponed to
Section~\ref{sec:enrichedproofs}. Note that, if we view the preimages
by $\psi$ of elements of $\mathrm{St}_{h,k}$ as equivalences classes
for an equivalence relation on $\mathcal{S}(H,K)$, then
\eqref{eq:intertwin} shows that the MJMC is the projection of the
enriched chain \cite[Lemma 2.5]{LevinPeresWilmer}. Noting that the sum
of each row of $\Psi$ is $1$, we immediately deduce:

\begin{cor} \label{cor:interstat} We have
  \begin{equation}
    \pi = \tilde{\pi} \Psi
  \end{equation}
  where $\tilde{\pi}$ is the stationary probability distribution of
  the enriched chain (viewed as a row vector), and $\pi$ is that of
  the MJMC.
\end{cor}

Our route to Theorem~\ref{thm:prob-juggle} is now clear. We want to find an
explicit expression for $\tilde{\pi}$, then compute $\tilde{\pi}
\Psi$. We first need to introduce some definitions and notations.

\begin{Def}
  \label{Def:arch}
  Let $\sigma$ be an enriched state and $s,t$ two integers such that
  $1 \leq s < t \leq H$. We say that the pair $(s,t)$ is an
  \emph{arch} of $\sigma$ if $s$ and $t$ belong to the same block
  $\beta$ of $\sigma$, while no integer strictly between $s$ and $t$
  belongs to $\beta$.
\end{Def}

Note that $\sigma \in \mathcal{S}(H,K)$ has exactly $\ell=H-K$ arches.
\begin{Not}
For $1 \leq s < t \leq H$, we denote by $C_\sigma(s,t)$ the number of
blocks containing at least one element in $\{s,s+1,\ldots,t\}$ (when $(s,t)$
is an arch, then we say that these blocks are those \emph{covered} by
$(s,t)$). 
\label{ZCsigma}
\end{Not}

We are now ready to express the stationary distribution of
the enriched chain, see Figure~\ref{fig:chain42} again for an
illustration in the case $H=4$, $K=2$.

\begin{lem} \label{lem:tildewdef}
  For $\sigma \in \mathcal{S}(H,K)$, the monomial
  \begin{equation} \label{eq:tildewdef}
    \tilde{w}(\sigma) = \prod_{(s,t) \text{ arch of } \sigma} x_{K-C_\sigma(s,t)}
  \end{equation}
  defines an unnormalized stationary measure of the enriched chain.
\end{lem}

\begin{rem}
  \label{rem:notarem}
  It is here natural to introduce the notation $z_i=x_{K-i}$,
  $i=1,\ldots,K$, so that $z_i$ is simply the weight per arch covering
  $i$ blocks. In the juggling language, $z_i$ is the probability of
  doing an insertion at the $i$-th available site starting from the
  right.  This notation will be useful for the ``add-drop'' and
  ``annihilation'' generalizations of the MJMC, see
  Sections~\ref{subsec:juggle-adddrop} and
  \ref{subsec:juggle-annihilation}. However, the $x_i$ notation is
  more convenient to study the $k \to \infty$ limit, as we do
  in Sections~\ref{subsec:unb} and \ref{subsec:inf}.
\end{rem}

Noting that, for $\tilde{\mathcal{E}}$ as in \eqref{eq:tildeEdef}, we
have $\tilde{w}(\tilde{\mathcal{E}})=x_0^\ell$, we find that
\begin{equation}
  Z_{h,k}=\sum_{\sigma \in \mathcal{S}(H,K)} \tilde{w}(\sigma) \label{eq:Zsetdef}
\end{equation}
is positive whenever $x_0>0$, so that $\tilde{\pi}=\tilde{w}/Z_{h,k}$
is the stationary probability distribution of the enriched chain.  For
$B \in \mathrm{St}_{h,k}$, we set
\begin{equation}
  w(B) = \sum_{\sigma \in \psi^{-1}(B)} \tilde{w}(\sigma),
\end{equation}
that is $w=\tilde{w}\Psi$. This implies that  $\pi=w/Z_{h,k}$ is the stationary
probability distribution of the MJMC by Corollary~\ref{cor:interstat}.
Theorem~\ref{thm:prob-juggle} then follows immediately from:

\begin{lem}
  \label{lem:preimag}
  For $B=b_1 \cdots b_h \in \mathrm{St}_{h,k}$, we have
  \begin{equation} \label{eq:preimag}
    w(B) = \prod_{\substack{i\in \{1,\ldots,h\} \\ b_i =
        \bullet}} \left( x_{E_i(B)} + \cdots + x_k \right).
  \end{equation}
  where $E_i(B) = \# \{j<i | b_j = \circ \}$.
\end{lem}

\begin{exmp}
  Returning again to the case $H=4$, $K=2$ (i.e.\ $h=3$, $k=1$)
  illustrated on Figure~\ref{fig:chain42}, we find
  \begin{equation}
    w(\circ\bullet\bullet) = x_1^2, \qquad
    w(\bullet\circ\bullet) = (x_0+x_1)x_1=x_1, \qquad
    w(\bullet\bullet\circ) = (x_0+x_1)^2=1.
  \end{equation}
\end{exmp}

\begin{rem}
  The unnormalized weights $\tilde{w}(\sigma)$ are well-defined for
  arbitrary $x_0,\ldots,\allowbreak x_k$ (not necessarily positive or subject to
  the condition $x_0+\cdots+x_k=1$). For $x_0=\cdots=x_k=1$, we recover
  Warrington's combinatorial proof of the identity \eqref{eq:Zunif}.
  For $x_i=q^{k-i}$, we have $\tilde{w}(\sigma)=q^{N(\sigma)}$ with
  \begin{equation}
    N(\sigma)=\sum_{(s,t) \text{ arch of } \sigma} (C_\sigma(s,t)-1)\label{eq:Ndef}.
  \end{equation}
  The identity \eqref{eq:Zqanal} shows that $N(\sigma)$ is a so-called
  Mahonian statistic on set partitions, which seems different from the
  inversion number and the major index \cite{Sagan}.
\end{rem}

\subsubsection{Technical proofs}
\label{sec:enrichedproofs}

We now prove Lemmas \ref{lem:intertwin}, \ref{lem:tildewdef} and
\ref{lem:preimag}.

\begin{proofof}{Lemma~\ref{lem:intertwin}}
  We check the relation \eqref{eq:intertwin} coefficientwise.  Fix
  $\sigma \in \mathcal{S}(H,K)$ and $B \in \mathrm{St}_{h,k}$, and let
  $A = a_1 \cdots a_h = \psi(\sigma)$, where $\psi(\sigma)$ is defined in the beginning of Section~\ref{sec:enricheddef}. Observe that
  \begin{equation} \label{eq:psicond1}
    \psi(\sigma^\downarrow \cup \{H\})=a_2 \cdots a_h \circ
  \end{equation}
  and that, for all $i$,
  \begin{equation} \label{eq:psicond2}
    \psi(I_{i}(\sigma^\downarrow))=T_i(a_2 \cdots a_h \circ)
  \end{equation}
  (since inserting $H$ into the $(i+1)$-th block of
  $\sigma^\downarrow$ transforms its maximum into a non-maximum,
  which means replacing the $(i+1)$-th occurrence of $\circ$ in
  $\psi(\sigma^\downarrow)\circ=a_2 \cdots a_h \circ$ by a
  $\bullet$).  If $\{1\} \in \sigma$ (i.e.\ $a_1=\circ$), then, by the
  definition of the transition matrices and \eqref{eq:psicond1},
  \begin{equation}
    (\tilde{P} \Psi)_{\sigma,B} =
    \delta_{\psi(\sigma^\downarrow) \cup \{H\}),B} =
    \delta_{a_2 \cdots a_h \circ,B} = P_{A,B} = (\Psi P)_{\sigma,B}.
  \end{equation}
  Otherwise, if $\{1\} \notin \sigma$ (i.e.\ $a_1=\bullet$), we may write
  \begin{equation}
    \begin{split}
      (\tilde{P} \Psi)_{\sigma,B} &= \sum_{\substack{\tau \in
          \mathcal{S}(H,K) \\ i \in \{0,\ldots,k\}}} x_i
      \delta_{I_i(\sigma^\downarrow),\tau} \delta_{\psi(\tau),B} =
      \sum_{i=0}^k
      x_i \delta_{\psi(I_i(\sigma^\downarrow)),B} \\
      &= \sum_{i=0}^k x_i \delta_{T_i(a_2 \cdots a_h \circ),B} = P_{A,B}
      = (\Psi P)_{\sigma,B} \\
    \end{split}
  \end{equation}
  where we use \eqref{eq:psicond2} to go from the first to the second
  line. In both cases, we have the wanted relation.
\end{proofof}

\begin{proofof}{Lemma~\ref{lem:tildewdef}}
  We need to show that
  \begin{equation} \label{eq:enrichSC} \tilde{w}(\sigma) = \sum_{\tau
      \in \mathcal{S}(H,K)} \tilde{w}(\tau) \tilde{P}_{\tau,\sigma}
  \end{equation}
  which is done by considering the possible predecessors of $\sigma$
  in the enriched chain. If $\{H\} \in \sigma$ then $\sigma$ has a
  unique precedessor $\tau$ such that $\sigma = \tau^\downarrow \cup
  \{H\}$. The arches of $\sigma$ and $\tau$ are clearly in one-to-one
  correspondence, and their weights are unchanged, thus
  $\tilde{w}(\sigma)=\tilde{w}(\tau)$ as wanted. Otherwise, there
  exists a unique pair $(\rho,i) \in \mathcal{S}(h,K) \times
  \{0,\ldots,k\}$ such that $\sigma=I_i(\rho)$ and it is easily seen
  that $\tilde{w}(\sigma)=x_i \tilde{w}(\rho)$ (since inserting $H$ in
  the $(i+1)$-th block of $\rho$ amounts to creating an arch covering
  $K-i$ blocks, the other arch weights being unaffected). The
  predecessors of $\sigma$ are the $\tau$ such that
  $\tau^\downarrow=\rho$ and $\{1\} \notin \tau$, and we have
  $\tilde{P}_{\tau,\sigma}=x_i$ regardless of $\tau$. There are
  exactly $K$ predecessors, whose weights are $x_j \tilde{w}(\rho)$
  with $j \in \{0,\ldots,k\}$ (indeed all these predecessors are
  obtained by shifting all elements of all blocks of $\rho$ up by one,
  which preserves the arch weights, then inserting $1$ into one of the
  $K$ blocks, which creates a new arch covering an arbitrary number of
  blocks between $1$ and $K$). From the condition $x_0+\ldots+x_k=1$,
  we conclude that \eqref{eq:enrichSC} holds as wanted.
\end{proofof}

\begin{proofof}{Lemma~\ref{lem:preimag}}
  We proceed by induction on $\ell=h-k=H-K$. For $\ell=0$, the
  statement is true since $B=\circ^h$ has only one preimage, namely
  the set partition consisting of singletons 
  $\{\{1\},\{2\},\ldots,\{H\}\}$ which has no arch,
  thus has a weight $1$ consistent with \eqref{eq:preimag}. Let us
  now assume that $\ell>0$ and consider the smallest $i$ such that
  $b_i=\bullet$, i.e.\ $B=\circ^{i-1}\bullet b_{i+1} \ldots b_h$. Let
  $B' \in \mathrm{St}_{h,k+1}$ be the word obtained by replacing the
  $i$-th letter of $B$ by $\circ$, i.e. $B'=\circ^i b_{i+1} \ldots
  b_h$.

  Consider $\tau \in \psi^{-1}(B)$: it is easily seen from the
  definition of $\psi$ that, for all $j<i$, the singleton $\{j\}$ is
  necessarily a block of $\tau$, but the block $\beta$ containing $i$
  contains at least another (larger) element, thus $\tau$ has an arch
  of the form $(i,i')$ with weight $\tilde{w}_{i,i'}=x_n$ for some $n
  \geq E_i(B)$.

  Splitting $\beta$ in two blocks $\{i\}$ and $\beta \setminus \{i\}$,
  we obtain a set partition $\tau' \in \mathcal{S}(H,K+1)$, and a
  moment's thought shows that $\tau \mapsto (\tau',n)$ is a bijection
  between $\psi^{-1}(B)$ and $\psi^{-1}(B') \times
  \{E_i(B),\ldots,k\}$ such that $\tilde{w}(\tau)=x_n
  \tilde{w}(\tau')$. By the induction hypothesis we have
  \begin{equation}
    w(B') = \sum_{\tau' \in \psi^{-1}(B')} \tilde{w}(\tau') =
    \prod_{\substack{j\in \{i+1,\ldots,h\} \\ b_j = \bullet}}
    \left( x_{E_j(B)} + \cdots + x_k \right).
  \end{equation}
  Multiplying this relation by $x_n$ and summing over $n$, the
  desired relation \eqref{eq:preimag} follows.
\end{proofof}

\section{Extensions to infinite state spaces}
\label{sec:infextensions}

In this section, we discuss extensions of the
Multivariate Juggling Markov Chain to an infinite setting. More precisely,
we first let the number of available sites $k$ tend to infinity,
keeping the number of balls $\ell$ fixed. This is the Unbounded MJMC
discussed in Section \ref{subsec:unb}. Further, we 
consider the case where the number of balls $\ell$ tends to infinity
(with $k$ infinite or not), which corresponds to the Infinite MJMC
discussed in Section \ref{subsec:inf}.

\subsection{Unbounded heights}
\label{subsec:unb}

As suggested in the conclusion of \cite{Warrington}, a first natural
extension is to allow Magnus to throw balls arbitrarily high, so
that the ball flight times are unbounded. This corresponds to taking
the limit $h \to \infty$ of the MJMC, keeping the number of balls
$\ell$ fixed. In the particle picture, the sites are now labelled by
the set of positive integers, and exactly $\ell$ sites are occupied by
a particle. The time evolution is essentially unchanged: at each time
step, all particles are moved one site to the left, and if there was a
particle on the first site, it is reinserted at an available site
anywhere on the lattice. We keep the MJMC prescription of choosing the
$i$-th available site with probability $x_{i-1}$, but since there are
now infinitely many available sites, we have an infinite sequence
$(x_i)_{i \geq 0}$ of nonnegative real numbers such that
$\sum_{i=0}^\infty x_i=1$.

Formally, states can be viewed as infinite words on the alphabet
$\{\bullet,\circ\}$ containing exactly $\ell$ occurrences of
$\bullet$.  We denote by $\mathrm{St}^{(\ell)}$ the set of such states
(which can be viewed as the direct limit of the set sequence
$(\mathrm{St}_{h,h-\ell})_{h \geq \ell}$). For $A \in
\mathrm{St}^{(\ell-1)}$ and $i$ a nonnegative integer, we let $T_i(A)
\in \mathrm{St}^{(\ell)}$ be the word obtained by replacing the
$(i+1)$-th occurrence of $\circ$ in $A$ by $\bullet$. 

\begin{Def}
  Given a nonnegative integer $\ell$ and a sequence $(x_i)_{i \geq 0}$
  of nonnegative real numbers such that $\sum_{i=0}^\infty x_i=1$, the
  \emph{Unbounded Multivariate Juggling Markov Chain} (UMJMC) is the
  Markov chain on the state space $\mathrm{St}^{(\ell)}$ for which the
  transition probability from $A=a_1a_2a_3\cdots$ to $B$ reads
  \begin{equation}
    P_{A,B}=
    \begin{cases}
      1 & \text{if $a_1=\circ$ and $B=a_2 a_3\cdots$,} \\
      x_i & \text{if $a_1=\bullet$ and $B=T_i(a_2 a_3\cdots)$,} \\
      0 & \text{otherwise.}
    \end{cases}
    \label{eq:ujugp}
  \end{equation}
\end{Def}

Again, we might identify $\mathrm{St}^{(\ell)}$ with the set of
$\ell$-element subsets of $\mathbb{N}$ (corresponding to the positions
of the occupied sites). For two such subsets $S$ and $S'$, the UMJMC
transition probability from $S$ to $S'$ is still given by
\eqref{eq:MJMCsubset}. Our model is thus a particular case of the
general model considered in \cite{LeskelaVarpanen}.
The following proposition is an immediate extension of
Proposition~\ref{prop:finirred}.

\begin{prop} \label{prop:unbirred} If $x_0>0$, then the UMJMC has a
  unique closed communicating class all of whose states are
  aperiodic. Furthermore, if infinitely many $x_i$'s are nonzero, then
  the UMJMC is irreducible.
\end{prop}

\textbf{From now on, we assume that $x_0>0$.} Note that, if
$k=\sup\{i: x_i>0\}$ is finite, then no insertion is ever made at a
position larger than $k$. Thus, the states with particles at positions
larger than $k$ are transient, and upon removing them we recover the
finite MJMC on $\mathrm{St}_{\ell+k,k}$. In this degenerate case, all
the forthcoming statements remain true, but were already established
in Section~\ref{sec:juggle}.

We are again interested in the stationary distribution, but the fact
that the state space is now countably infinite requires a bit of
care. Still, we might extend Theorem~\ref{thm:prob-juggle} in the
following form.

\begin{thm} \label{thm:unb-prob-juggle}
  The unique invariant measure (up to constant of proportionality) of the UMJMC is given by
  \begin{equation}
    \label{eq:unb-prob-juggle}
    w(B) = \prod_{i\in \mathbb{N},\,b_i = \bullet}
    y_{E_i(B)}
  \end{equation}
  where $B=b_1 b_2 \cdots \in \mathrm{St}^{(\ell)}$, $E_i(B)
  = \# \{j<i | b_j = \circ \}$ and $y_m=\sum_{j=m}^\infty x_j$.
\end{thm}

We prove this theorem by a straightforward extension of the discussion
of Section~\ref{subsec:mjmcpar}. In the integer partition language, the
extension simply consists in lifting the bound on part sizes. We
denote by $\mathrm{Par}_{\infty, \ell}$ the set of integer partitions
with at most $\ell$ parts and unbounded part sizes (which we can view
as nonincreasing sequences of $\ell$ nonnegative integers). This set
is naturally identified with $\mathrm{St}^{(\ell)}$: given a state in
$\mathrm{St}^{(\ell)}$, we denote by
$s_1<s_2<\ldots<s_{\ell-1}<s_\ell$ the positions of $\bullet$'s, so that
the corresponding integer partition is
$(s_\ell-\ell,s_{\ell-1}-(\ell-1),\ldots,s_1-1)$. We let the
reader verify that, in the integer partition language, the UMJMC
transitions are still given by
\eqref{eq:part}. Theorem~\ref{thm:unb-prob-juggle} may then be
reformulated as follows.

\begin{thm} \label{thm:heightstat}
  The unique invariant measure (up to constant of proportionality) of the UMJMC is given by
  \begin{equation} \label{eq:unbwparform}
    w(\lambda) = \prod_{i=1}^\ell y_{\lambda_i}
  \end{equation}
  where $\lambda=(\lambda_1,\ldots,\lambda_\ell) \in
  \mathrm{Par}_{\infty, \ell}$.
\end{thm}

\begin{proof} We simply check that
  \begin{equation}
    \label{eq:unbstat}
    w(\lambda)=\sum_{\mu \in \mathrm{Par}_{\infty,\ell}} w(\mu) P_{\mu,\lambda}.
  \end{equation}
  Any predecessor $\mu$ of $\lambda$ is either:
  \begin{itemize}
  \item of the form
    $\mu=(\lambda_1+1,\ldots,\lambda_{\ell}+1)$ in which
    case we have $w(\mu) = \prod_{i=1}^{\ell} y_{\lambda_i+1} =
     \prod_{i=1}^\ell (y_{\lambda_i}-x_{\lambda_i})$ and
    $P_{\mu,\lambda} = 1$ ;
  \item of the form
    $\mu=(\lambda_1+1,\ldots,\lambda_{j-1}+1,\lambda_{j+1},\ldots,\lambda_{\ell},0)$
    for some $j \in \{1,\ldots,\ell\}$ in which case we have $w(\mu) =
    \prod_{i=1}^{j-1} (y_{\lambda_i}-x_{\lambda_i})
    \prod_{i=j+1}^{\ell} y_{\lambda_i}$ and $P_{\mu,\lambda} =
    x_{\lambda_j}$.
  \end{itemize}
  The verification of \eqref{eq:unbstat} is then done exactly as in
  \eqref{eq:finstatverif}.
\end{proof}

We have exhibited an invariant measure of the UMJMC which is clearly
$\sigma$-finite. We might wonder whether it is actually finite, so
that it may be normalized into a probability distribution.

\begin{prop}
  The invariant measure $w$ of the UMJMC is finite if and only if
  \begin{equation}
    \label{eq:fincond}
    \sum_{i=0}^\infty i x_i < \infty
  \end{equation}
  in which case its total mass reads
  \begin{equation}
    \label{eq:Zunb}
    Z^{(\ell)} = h_{\ell}(y_0,y_1,\ldots).
  \end{equation}
\end{prop}

\begin{proof}
  Let $Z^{(\ell)} \in [0,\infty]$ be the total mass of $w$. Observe
  that
  \begin{equation}
    Z^{(1)} = \sum_{\lambda_1=0}^\infty y_{\lambda_1} = \sum_{i=0}^\infty (i+1) x_i
  \end{equation}
  thus the first assertion is obviously true for $\ell=1$. It remains
  true for general $\ell$ as we have the inequalities
  \begin{equation}
    Z^{(1)} \leq Z^{(\ell)} \leq \left( Z^{(1)} \right)^\ell
  \end{equation}
  since $\mathrm{Par}_{\infty,\ell}$ contains the set of partitions
  with at most one nonzero part (which has mass $Z^{(1)}$ by
  \eqref{eq:unbwparform} and the fact that $y_0=1$), but may be viewed
  as a subset of $\mathbb{N}^\ell$, for which we have $\sum_{\lambda
    \in \mathbb{N}^\ell} w(\lambda) = \left( Z^{(1)} \right)^\ell$ by
  \eqref{eq:unbwparform}.

  The identity \eqref{eq:Zunb} is obtained by letting $k \to \infty$
  in \eqref{eq:Zparform}, since $(\mathrm{Par}_{k,\ell})_{k \geq 0}$
  forms an increasing family of sets with union
  $\mathrm{Par}_{\infty,\ell}$ so that $Z_{k+\ell,k} \nearrow
  Z^{(\ell)}$. Note that the notation $h_{\ell}(y_0,y_1,\ldots)$ makes
  sense: $h_\ell$, being a symmetric function, can be expressed as a
  polynomial in the power sum symmetric functions $(p_m)_{m \geq 1}$
  that does not depend on its number of variables, and
  \begin{equation}
    p_m(y_0,y_1,\ldots) = y_0^m + y_1^m + \cdots
  \end{equation}
  is a finite real number for any $m \geq 1$ by \eqref{eq:fincond}.
\end{proof}

By standard results from the theory of Markov chains, see
e.g. \cite[Chapter 13]{LeGall} or \cite[Chapter
21]{LevinPeresWilmer}, we deduce:

\begin{cor}
  The UMJMC is positive recurrent if and only if \eqref{eq:fincond}
  holds. In that case, there is a unique stationary probability
  distribution, and the chain started from any initial state converges
  to it in total variation as time tends to infinity.
\end{cor}

\begin{exmp}
  \label{exmp:unbgeom}
  Fix $q \in (0,1)$ and pick $x_i=(1-q)q^i$. We recover the ``JEP with
  memoryless height distribution'' with parameter $q$ considered in
  \cite{LeskelaVarpanen}. Clearly \eqref{eq:fincond} holds and we
  recover from Theorem~\ref{thm:unb-prob-juggle} Leskel\"a-Varpanen's
  expression for the stationary probability
  distribution. Interestingly, in the integer partition language, we
  find that $w(\lambda)=q^{|\lambda|}$ for
  $\lambda\in\mathrm{Par}_{\infty,\ell}$, where $|\lambda|$ stands for
  the size (sum of all parts) of $\lambda$.
\end{exmp}

\begin{rem}
  It is interesting to note that \eqref{eq:fincond} is a necessary and
  sufficient condition for the UMJMC to be uniformly integrable in the
  sense of \cite{LeskelaVarpanen}. When \eqref{eq:fincond} does not
  hold, the chain may either be null recurrent or transient. For
  $\ell=1$, the state with the first site occupied is clearly
  recurrent, thus the chain is null recurrent. Figuring out the
  situation for $\ell > 1$ is an intriguing question. By somewhat
  heuristic arguments inspired by Poly\'a's theorem on recurrence of random walks, we expect that,
  for a sequence $x_i$ decaying asymptotically as $i^{-1-\alpha}$ with
  $\alpha \in (0,1)$ (so that $\sum i x_i=\infty$), the UMJMC is null
  recurrent if $\ell(1-\alpha) \leq 1$ and transient if $\ell
  (1-\alpha)>1$. A formal proof of this statement is beyond the scope
  of this paper.
\end{rem}

\subsection{Infinitely many balls}
\label{subsec:inf}

We now consider another limit in which Magnus juggles with infinitely
many balls ($\ell \to \infty$). Let us first heuristically discuss
this limit in the particle picture, a precise mathematical statement
(Proposition~\ref{prop:imjmcconv}) being given at the end of this section.
Since the particles tend to
accumulate on the left side of the lattice (i.e.\ balls tend to have
low remaining flight times), we expect that, as $\ell$ becomes large,
all sites at a finite distance from the first site will be occupied
with high probability (in particular, Magnus receives a ball at every
time step), and the first available site will then be typically at a
position of order $\ell$. It is convenient to relabel the sites by
arbitrary (not necessarily positive) integers, so that the leftmost
site has label $-\ell$, and the lowest state corresponds to having
particles on sites $-\ell,-\ell+1,\ldots,-1$. In the limit $\ell \to
\infty$, the leftmost site is sent to $-\infty$, and the lowest state
corresponds to having all negative sites filled and all nonnegative
sites empty.

A general state is obtained by moving finitely many particles from
occupied to empty sites, in other words it is a bi-infinite word
$B=(b_i)_{i \in \mathbb{Z}} \in \{\bullet,\circ\}^\mathbb{Z}$ such
that
\begin{equation}
  \label{eq:infbal}
  \{ i \geq 0: b_i=\bullet \} = \{ i<0: b_i=\circ \} < \infty
\end{equation}
(such a word is sometimes called a ``Maya diagram''). The time
evolution is now easy to describe: at each time step, all particles
are moved one site to the left and a new particle is inserted at an
available site (so that the condition \eqref{eq:infbal} is preserved).
Again, we keep the MJMC prescription of picking the $i$-th available
site from the left with probability $x_{i-1}$ (this is well-defined
since the set of available positions is bounded from below).

Rather than writing down the transitions formally, we prefer to work
directly in the language of integer partitions. It is well-known that
bi-infinite words subject to \eqref{eq:infbal} are in one-to-one
correspondence with arbitrary integer partitions (i.e.\ non-increasing
sequences of integers that vanish eventually), see again
Figure~\ref{fig:bijection_states_partitions} and think about extending
the displayed juggling state by adding infinitely many $\bullet$'s on
the left and infinitely many $\circ$'s on the right, which does not
change the corresponding integer partition. We denote by
$\mathrm{Par}$ the set of all integer partitions and, for
$\lambda=(\lambda_j)_{j\geq 1} \in \mathrm{Par}$ and $i \geq 0$, we
set
\begin{equation}
  \label{eq:inftrans}
  \lambda^{(i)}=(\lambda_1-1,\ldots,\lambda_{j-1}-1,i,\lambda_j,\lambda_{j+1},\ldots)
\end{equation}
with $j$ the smallest index such that $\lambda_j\leq i$. This is the
$\mathrm{Par}$ equivalent of the above time evolution with insertion
at the $(i+1)$-th available site.

\begin{Def}
  \label{Def:inf}
  Given a sequence $(x_i)_{i \geq 0}$ of nonnegative real numbers such
  that $\sum_{i=0}^\infty x_i=1$, the \emph{Infinite Multivariate
    Juggling Markov Chain} (IMJMC) is the Markov chain on the state
  space $\mathrm{Par}$ for which the transition probability from
  $\lambda$ to $\mu$ reads
  \begin{equation}
    \label{eq:infdef}
    P_{\lambda,\mu} =
    \begin{cases}
      x_i & \text{if $\mu=\lambda^{(i)}$,}\\
      0 & \text{otherwise.}
    \end{cases}
  \end{equation}
\end{Def}

The existence and uniqueness (up to normalization) of an invariant
measure of the IMJMC is ensured by the following:

\begin{prop}
  If $x_0>0$, then the IMJMC has a unique closed communicating class,
  whose all states are aperiodic. They are precisely the integer
  partitions whose parts are smaller than or equal to $\sup \{i: x_i>0\}$. In
  particular, if infinitely many $x_i$'s are nonzero, then the IMJMC is
  irreducible.
\end{prop}

\begin{proof}
  It is easily seen that, starting from any initial state and applying
  repeatedly the transition $x_0$, we end up with the empty partition
  (i.e.\ the partition corresponding to an empty Young diagram), which
  is a fixed point. This proves the first statement. For the second
  statement, let us prove that, for any $k$ such that $x_k>0$, we may
  obtain any partition $\mu$ with $\mu_1 \leq k$ from the empty
  partition using only the transitions $x_0$ and $x_k$. We proceed by
  double induction on the number of parts $p$ of $\mu$ (i.e.\ the
  largest $p$ such that $\mu_p>0$) and on $k-\mu_1$. If $p=0$, then
  $\mu$ is already the empty partition. Otherwise, if $k-\mu_1=0$,
  then $\mu$ is obtained from $(\mu_2,\mu_3,\ldots)$ (which has one
  less part) by the transition $x_k$. Finally, if $k-\mu_1>0$, then
  $\mu$ is obtained from $\mu'=(\mu_1+1,\ldots,\mu_p+1,0,0,\ldots)$
  (which still has $p$ parts and $k-\mu'_1<k-\mu_1$) by the transition
  $x_0$.
\end{proof}

\begin{thm}
  \label{thm:heightstatinf}
  An invariant measure of the IMJMC is given by
  \begin{equation} \label{eq:infwparform}
    w(\lambda) = \prod_{i=1}^\infty y_{\lambda_i}
  \end{equation}
  where $\lambda \in \mathrm{Par}$ and $y_m=\sum_{j=m}^{\infty}
  x_i$. Its total mass reads
  \begin{equation}
    \label{eq:Zinf}
    Z = \prod_{m=1}^\infty \frac{1}{1-y_m}
  \end{equation}
  and $Z$ is finite (i.e.\ the IMJMC is positive recurrent) if and
  only if \eqref{eq:fincond} holds.
\end{thm}

\begin{proof}
  Note that the right-hand side of \eqref{eq:infwparform} is well-defined
  since $y_0=1$ and $\lambda_i$ vanishes eventually.  By
  \eqref{eq:inftrans}, any predecessor of $\lambda \in \mathrm{Par}$
  is necessarily of the form
  $\mu=(\lambda_1+1,\ldots,\lambda_{j-1}+1,\lambda_{j+1},\lambda_{j+2},\ldots)$
  for some $j$, and $\lambda=\mu^{(\lambda_j)}$. We then have $w(\mu)
  = \prod_{i=1}^{j-1} (y_{\lambda_i}-x_{\lambda_i})
  \prod_{i=j+1}^{\infty} y_{\lambda_i}$ and $P_{\mu,\lambda} =
  x_{\lambda_j}$. Telescoping as in \eqref{eq:finstatverif}, we deduce
  \begin{equation} \label{eq:infstatverif}
    \begin{split}
      \sum_{\mu \in \mathrm{Par}} w(\mu) P_{\mu,\lambda} &=
      \sum_{j=1}^{\infty} \left(x_{\lambda_j}
        \prod_{i=1}^{j-1}(y_{\lambda_i}-x_{\lambda_i})\prod_{i=j+1}^{\infty}
        y_{\lambda_i}\right) = \prod_{i=1}^\infty y_{\lambda_i} = w(\lambda)
    \end{split}
  \end{equation}
  (note that
  $\prod_{i=1}^{j-1}(y_{\lambda_i}-x_{\lambda_i})\prod_{i=j+1}^{\infty}
  y_{\lambda_i} \leq y_1^{j-1}$ with $y_1<1$). This establishes the
  invariance of $w$.

  The expression \eqref{eq:Zinf} follows from standard considerations
  on integer partitions, and is finite if and only if $\sum y_m <
  \infty$, i.e.\ if \eqref{eq:fincond} holds.
\end{proof}

Again, when \eqref{eq:fincond} does not hold, the IMJMC is either null
recurrent or transient, and it would be interesting to know which
situation occurs.

\begin{exmp} \label{exmp:infgeom}
  Consider again the geometric case $x_i=(1-q)q^i$, $q \in
  (0,1)$. Then, the stationary distribution is nothing but the
  ``$q^{\mathrm{size}}$'' measure over arbitrary integer
  partitions. Note that, contrary to the case of finitely many balls
  (Example~\ref{exmp:unbgeom}), the phenomenon of ultrafast
  convergence to stationarity observed in \cite{LeskelaVarpanen} cannot
  occur: the stationary distribution is supported on the set of all
  integer partitions with arbitrarily many parts, but since we may
  create at most one new part at each time step, the distribution at a
  finite time starting from a given initial partition is supported on
  a strictly smaller set.
\end{exmp}

So far we have not given a precise mathematical meaning to the fact that
the IMJMC is the limit as $\ell \to \infty$ of the (U)MJMC. This is
actually the case according to a certain notion of ``local
convergence'', which again is easier to state in the language of
integer partitions. Let us fix the sequence $(x_i)_{i \geq 0}$ and,
for $\ell \in \mathbb{N} \cup \{\infty\}$, denote by
$\Lambda(\ell;0),\Lambda(\ell;1),\ldots$ the $\mathrm{Par}$-valued
Markov chain corresponding to the UMJMC with $\ell$ balls for $\ell
\in \mathbb{N}$, or the IMJMC for $\ell = \infty$, started at an
arbitrary (deterministic or random) initial state $\Lambda(\ell,0)$.
(Note that the state space $\mathrm{Par}_{\infty, \ell}$ of
$\Lambda(\ell,\cdot)$ with $\ell \in \mathbb{N}$ may naturally be
viewed as a subset of $\mathrm{Par}$, upon appending an infinite zero
sequence to its elements.)

\begin{prop}
  \label{prop:imjmcconv}
  For any nonnegative integer $t$ and any fixed partition $\nu$,
  we have the convergence in distribution
  \begin{equation}
    \label{eq:imjmcconv}
    \left(\Lambda(\ell;0),\ldots,\Lambda(\ell;t)\right)
    \xrightarrow[\ell \to \infty]{(d)}
    \left(\Lambda(\infty;0),\ldots,\Lambda(\infty;t)\right)
  \end{equation}
  when each chain is started at the deterministic state
  $\Lambda(\ell,0)=\nu$ (assuming that $\ell$ is larger than the
  number of parts of $\nu$).

  Furthermore, if the condition \eqref{eq:fincond} for positive
  recurrence is satisfied, then the convergence \eqref{eq:imjmcconv}
  also holds when each chain is started at its stationary probability
  distribution.
\end{prop}

\begin{proof}
  The transition probabilities for the chain $\Lambda(\ell,\cdot)$ are
  given by \eqref{eq:part} when $\ell \in \mathbb{N}$ and by
  \eqref{eq:infdef} when $\ell=\infty$: observe that they are equal
  whenever $\lambda$ has strictly less than $\ell$ parts, and
  furthermore that any transition increases the number of parts by at
  most one. It follows that, when each chain is started at the
  deterministic state $\Lambda(\ell,0)=\nu$, then the law of
  $\left(\Lambda(\ell;0),\ldots,\Lambda(\ell;t)\right)$ does not
  depend on $\ell$ as soon as $\ell-t$ is larger than the number of
  parts of $\nu$, which immediately implies the first statement.

  To prove the second statement we note that the stationary
  distribution of $\Lambda(\ell,\cdot)$ converges in total variation
  to that of $\Lambda(\infty,\cdot)$ as $\ell \to \infty$ (so the
  statement holds for $t=0$). This is a simple consequence of
  Theorems~\ref{thm:heightstat} and \ref{thm:heightstatinf}: the
  unnormalized invariant measure of $\Lambda(\ell,\cdot)$ for $\ell
  \in \mathbb{N}$ is the restriction of that of
  $\Lambda(\infty,\cdot)$ to the set of integer partitions with at
  most $\ell$ parts and, when \eqref{eq:fincond} holds, their total
  masses are finite and tend to one another as $\ell \to
  \infty$. Since a partition has finitely many parts, the previous
  argument allows us to conclude that \eqref{eq:imjmcconv} holds for any
  $t$.
\end{proof}

\begin{rem}
  It is easier to state the convergence in terms of integer partitions
  since the state spaces of the chains $\Lambda(\ell,\cdot)$ are
  naturally included in one another. In terms of particles, this
  corresponds to embedding the state space $\mathrm{St}^{(\ell)}$ of
  the UMJMC with $\ell$ balls into that of the IMJMC by prepending
  infinitely many $\bullet$'s on the left, then shifting all letters in
  the resulting bi-infinite word by $\ell$ positions to the left so as
  to satisfy \eqref{eq:infbal}. This allows us to translate
  Proposition~\ref{prop:imjmcconv} in the particle language, and
  justifies the heuristic discussion at the beginning of this section.
\end{rem}

\section{Extensions with a fluctuating number of balls}
\label{sec:flucext}

We now consider extensions of the MJMC where the number of balls is
not fixed but is allowed to fluctuate. These extensions are the
natural multivariate generalizations of the so-called \emph{add-drop}
and \emph{annihilation} models introduced in \cite[Section
4]{Warrington} (to which we refer for motivations), and we thus keep
the same denomination here. Both models are defined on the same state
space and have the same transition graph, only the transitions
probabilities differ. In both cases, we will provide an exact
expression for the stationary distribution, whose validity will be
proved by considering the enriched version of the chain.

Since many definitions and notations will be common to both models, we
factorize their discussion here. The basic state space of the model
will be $\mathrm{St}_h = \{\circ,\bullet\}^h$, with $h$ a fixed
nonnegative integer. As hinted in Remark~\ref{rem:notarem}, it will be
convenient here to read a word from right to left.

\begin{Not}
  \label{not:Si}
  For $A \in \bigcup_{h \geq 0} \mathrm{St}_h$ and $i$ a nonnegative
  integer, we let $S_i(A)$ be the word obtained by replacing the
  $i$-th occurrence of $\circ$ in $A$ by $\bullet$, upon reading the
  word \emph{from the right} (if $i$ is equal to $0$ or larger than
  the number of occurrences of $\circ$ in $A$ then we set $S_i(A)=A$
  by convention).
\end{Not}

Note that $S_i(A)=T_{k-i}(A)$ for $A \in \mathrm{St}_{h,k}$ and $1
\leq i \leq k$, with $T$ as in Section~\ref{subsec:juggle-def}. We now
turn to the enriched model, whose state space is the set
$\mathcal{S}(H)$ of all set partitions of $\{1,\ldots,H\}$, with
$H=h+1$ (the number of enriched states is thus a Bell number). Recall
the notations from Section~\ref{sec:enricheddef}: $\psi$ is a
surjection from $\mathcal{S}(H)$ onto $\mathrm{St}_h$ and $\sigma
\mapsto \sigma^\downarrow$ is a mapping from $\mathcal{S}(H)$ to
$\mathcal{S}(h)$.

\begin{Not}
  \label{not:Ji}
  For $\tau \in \bigcup_{h \geq 0} \mathcal{S}(h)$ and $i$ a
  nonnegative integer, we define $J_i(\tau)$ as follows. If $\tau \in
  \mathcal{S}(h)$ then we let $J_i(\tau) \in \mathcal{S}(h+1)$ be the
  set partition obtained by inserting $h+1$ into the $i$-th block of
  $\tau$, now numbered by \emph{decreasing} order of maxima (if $i$ is
  equal to $0$ or larger than the number of blocks of $\tau$ then we
  set $J_i(\tau)=\tau \cup \{h+1\}$ by convention).
\end{Not}

Note that $J_i(\tau)=I_{K-i}(\tau)$ for $\tau \in \mathcal{S}(h,K)$
and $1 \leq i \leq K$. Observe that, for any $i\geq 0$, we have the
fundamental ``intertwining'' relation
\begin{equation}
  \label{eq:mapintertwin}
  \psi(J_i(\tau)) = S_i(\psi(\tau)\circ)
\end{equation}
which is nothing but a compact rewriting of \eqref{eq:psicond1} and
\eqref{eq:psicond2}. Note also that, when $\tau$ is non empty ($h \geq
1$), we have the commutation relation
\begin{equation}
  \label{eq:mapcommut}
  J_i(\tau)^\downarrow = J_i(\tau^\downarrow).
\end{equation}

The basic transition graph is defined as follows: for any
$A=a_1a_2\ldots a_h \in \mathrm{St}_h$ and $i \geq 0$, we have an
oriented edge from $A$ to $S_i(a_2\ldots a_h\circ)$ (we ignore edge
multiplicities). See Figure \ref{chain2} for $h=2$. Similarly, the enriched transition graph is obtained
by connecting each $\sigma \in \mathcal{S}(H)$ to
$J_i(\sigma^\downarrow)$ for all $i \geq 0$. It is not difficult to
check that both transition graphs are strongly connected using Proposition~\ref{prop:finirred} and Proposition~\ref{prop:enrichirred}.

\begin{figure}[htpb]
  \centering
  \includegraphics[width=.7\textwidth]{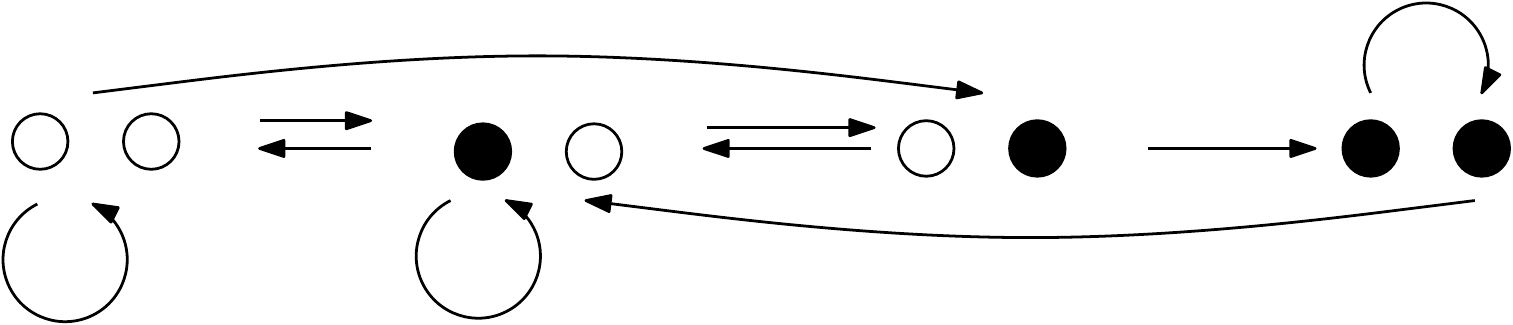}
  \caption{The basic transition graph for $h=2$}
  \label{chain2}
\end{figure}

\subsection{Add-drop juggling}
\label{subsec:juggle-adddrop}

\begin{Def} \label{Def:ADJMC} Given $h$ a nonnegative integer and $a
  =z_0,\ldots,z_h$ nonnegative real numbers, the (multivariate)
  \emph{add-drop model} is the Markov chain on the state space
  $\mathrm{St}_h$ for which the transition probability from
  $A=a_1\cdots a_h$ to $B$ reads
  \begin{equation}
    P_{A,B} = \left\{
    \begin{array}{cl}
      \displaystyle \frac{z_i}{z_0+\cdots+z_k} & \text{
        if $B=S_i(a_2\ldots a_h \circ)$ for some $i \in \{0,\ldots,k\}$,}\\
      0 & \text{otherwise,}
    \end{array} \right.
    \label{eq:adjmc}
  \end{equation}
  with $k$ the number of occurrences of $\circ$ in $a_2\ldots a_h
  \circ$.
\end{Def}

Warrington's add-drop model is recovered by taking
$a=z_1=\cdots=z_h=1$.  It is easily seen that, in general, the chain
is aperiodic with a unique communicating class whenever $a>0$, so that
the stationary distribution is unique.

\begin{thm} \label{thm:prob-juggle-adddrop} 
  The stationary distribution of the add-drop model is given by
  \begin{equation}
    \label{prob-juggle-adddrop}
    \Pi(B) = \frac{a^k}{Z_{h}} \prod_{\substack{i=1 \\ b_i = \bullet}}^h 
    \left( z_1 + \cdots + z_{\psi_i(B)+1} \right),
  \end{equation}
  for $B = b_1 \ldots b_h \in \mathrm{St}_{h,k}$, with $\psi_i(B) =
  \# \{j: i<j\leq h, \, b_j = \circ\}$ and
  \begin{equation}
    Z_h = \sum_{k=0}^h a^k
    h_{h-k}(z_1,z_1+z_2,\ldots,z_1+\cdots+z_{k+1}),
  \end{equation}
  where $h_\ell$ is the complete homogeneous symmetric polynomial of
  degree $\ell$. (Note that $z_{h+1}$ never appears in
  \eqref{prob-juggle-adddrop} since the product is empty for $k=h$.)
\end{thm}

\begin{exmp} \label{exmp:juggle-adddrop}
The transition matrix for $h=2$ in the lexicographically-ordered basis 
$(\;\bullet\bullet , \; \circ\bullet, \; \bullet\circ, \; \circ\circ)$
is given by
\begin{equation}
P =   \left( 
 \begin {array}{cccc} \noalign{\medskip}\ds {\frac {z_{{1}}}{a+
z_{{1}}}}&0&\ds {\frac {a}{a+z_{{1}}}}&0\\
\noalign{\medskip}\ds {\frac {z_{{1}}}{a+
z_{{1}}}}&0&\ds {\frac {a}{a+z_{{1}}}}&0\\
0 &\ds {\frac {z_{{1}}}{a+z_{{1}}+z_{{2}}}}&\ds {\frac {z_{{2}}}{a+z_{{1}}+z_{{2}}}}&\ds {\frac {a}{a+z_{{1}}+z_{{2}}}}\\
 0 &\ds {\frac {z_{{1}}}{a+z_{{1}}+z_{{2}}}}&\ds {\frac {z_{{2}}}{a+z_{{1}}+z_{{2}}}}&\ds {\frac {a}{a+z_{{1}}+z_{{2}}}}\\
\end {array} 
\right) ,
\end{equation}
and the stationary probabilities are given by the normalized coordinates of the row eigenvector with eigenvalue 1, namely  
\begin{equation}
\frac 1{Z_2}
\left( {z_1^{2}}, \;
{az_{{1}}}, \;
{a(z_{{1}}+z_{{2}})}, \;
{a^{2}} 
\right),
\end{equation}
where $Z_2 = {a}^{2}+az_{{2}}+2\,az_{{1}}+z_{1}^{2}$.
\end{exmp}

Theorem~\ref{thm:prob-juggle-adddrop} is the natural multivariate
generalization of \cite[Theorem 3, item 1]{Warrington}. The
parameter $a$ has the physical interpretation of the fugacity for a
$\circ$. In other words the distribution $\Pi$ is the grand-canonical
version of the stationary distribution $\pi$ of
Theorem~\ref{thm:prob-juggle}.
We prove Theorem~\ref{thm:prob-juggle-adddrop} by a straightforward
extension of the construction of Section~\ref{subsec:enriched}, which
we now detail.

\begin{Def} \label{Def:EADMC} The (multivariate) \emph{enriched
    add-drop model} is the Markov chain on the state space
  $\mathcal{S}(H)$ for which the transition probability from $\sigma$
  to $\tau$ is given by
  \begin{equation}
    \tilde{P}_{\sigma,\tau} = \left\{
    \begin{array}{cl}
      \displaystyle \frac{z_i}{z_0+\cdots+z_k} & \text{if $\tau=J_i(\sigma^\downarrow)$ for some $i \in \{0,\ldots,k\}$,}\\
      0 & \text{otherwise,}
    \end{array} \right.
   \label{eq:enriched-add-drop-matrix}
  \end{equation}
  with $k$ the number of blocks of $\sigma^\downarrow$.
\end{Def}

By \eqref{eq:mapintertwin} it is immediate that the add-drop model is
indeed the projection of the enriched chain, and furthermore it is
easily seen that the enriched chain is aperiodic with a unique
communicating class for
$z_0=a>0$.

\begin{lem} \label{lem:prob-setp-adddrop}
  For $\sigma \in \mathcal{S}(H)$ with, say, $K$ blocks, the monomial
  \begin{equation}
    \tilde{W}(\sigma) = a^{K-1} \prod_{(s,t) \text{ arch of } \sigma} z_{C_{\sigma}(s,t)},\label{eq:adddropstat}
  \end{equation}
  where $C_{\sigma}(s,t)$ is as in Notation~\ref{ZCsigma}, defines
  an unnormalized stationary distribution of the enriched chain.
\end{lem}

\begin{proof}
  We want to show that $\tilde{W}(\tau) = \sum_{\sigma \in
    \mathcal{S}(H)} \tilde{W}(\sigma)\tilde{P}_{\sigma,\tau}$ for any
  $\tau$. Let $\tau^{\uparrow}$ be the set partition obtained by
  removing the element $H$ from $\tau$, so that the possible predecessors of
  $\tau$ are the $\sigma$ such that $\sigma^{\downarrow} =
  \tau^{\uparrow}$. More precisely, any predecessor is obtained by
  shifting all elements of $\tau^\uparrow$ up by $1$, then inserting
  the element $1$ either as a singleton or into a preexisting block.
  Hence, there are $k+1$ predecessors where $k$ is the number of
  blocks of $\tau^\uparrow$, and we readily see that their weights are
  $z_i \tilde{W}(\tau^\uparrow)$ with $i=0,\ldots,k$. On the other
  hand, let $j$ be equal to $0$ if $\{H\}$ is a singleton of $\tau$,
  or equal to $C_\tau(r,H)$ where $r$ is such that $(r,H)$ is an arch
  of $\tau$. We have $\tilde{W}(\tau)=z_j \tilde{W}(\tau^\uparrow)$
  and $\tilde{P}_{\sigma,\tau}=z_j/(z_0+\cdots+z_k)$ for any predecessor
  $\sigma$. We deduce that, as wanted,
  \begin{equation}
    \sum_{\sigma \in
      \mathcal{S}(H)} \tilde{W}(\sigma)\tilde{P}_{\sigma,\tau} =
    \sum_{i=0}^k z_i \tilde{W}(\tau^\uparrow) \frac{z_j}{z_0+\cdots+z_k}
    = \tilde{W}(\tau).
  \end{equation}
\end{proof}

\begin{exmp} \label{exmp:setp-adddrop}
The transition matrix for $H=3$ in the ordered basis\\
$(1,2,3; \;\; 1\,|\,2,3; \;\; 2\,|\,1,3; \;\; 1,2\,|\,3; \;\; 1\,|\, 2\,|\, 3)$
is given by
\begin{equation}
\tilde{P} =   \left( 
 \begin {array}{ccccc} 
\ds {\frac {z_{{1}}}{a+z_{{1}}}}&0&0&\ds {\frac {a}{a+z_{{1}}}}&0\\
\noalign{\medskip}\ds {\frac {z_{{1}}}{a+z_{{1}}}}&0&0&\ds {\frac {a}{a+z_{{1}}}}&0\\
\noalign{\medskip}0&\ds {\frac {z_{{1}}}{a+z_{{1}}+z_{{2}}}}&\ds {\frac {z_{{2}}}{a+z_{{1}}+z_{{2}}}}&0&\ds {\frac {a}{a+z_{{1}}+z_{{2}}}}\\
\noalign{\medskip}0&\ds {\frac {z_{{1}}}{a+z_{{1}}+z_{{2}}}}&\ds {\frac {z_{{2}}}{a+z_{{1}}+z_{{2}}}}&0&\ds {\frac {a}{a+z_{{1}}+z_{{2}}}}\\
\noalign{\medskip}0&\ds {\frac {z_{{1}}}{a+z_{{1}}+z_{{2}}}}&\ds {\frac {z_{{2}}}{a+z_{{1}}+z_{{2}}}}&0&\ds {\frac {a}{a+z_{{1}}+z_{{2}}}}
\end {array} 
\right) ,
\end{equation}
and admits the row vector $\left( {z_{1}^{2}}, \; {az_{{1}}}, \;
  {az_{{2}}}, \; {az_{{1}}}, \; {{a}^{2}} \right)$ as a left
eigenvector of eigenvalue $1$. Compare this with the stationary distribution in Example~\ref{exmp:juggle-adddrop}.
\end{exmp}

\begin{proofof}{Theorem~\ref{thm:prob-juggle-adddrop}}
  For $B \in \mathrm{St}_{h,k}$, let
  \begin{equation}
    W(B)=\sum_{\sigma \in \psi^{-1}(B)} \tilde{W}(\sigma).
  \end{equation}
  By Lemma~\ref{lem:preimag} and Remark~\ref{rem:notarem}
  (note that $x_{E_i(B)}=z_{\psi_i(B)+1}$), we readily deduce
  \begin{equation}
    W(B)=a^k \prod_{\substack{i=1 \\ b_i =
        \bullet}}^h \left( z_1 + \cdots + z_{\psi_i(B)+1} \right).
  \end{equation}
  The total mass of $W$ is clearly equal to $\sum_{k=0}^h a^k
  Z_{h,k}(z_{k+1},z_k,\ldots,z_1)$ with $Z_{h,k}$ the normalization
  factor of the MJMC defined in Proposition \ref{prop:Zhform}.
\end{proofof}

\subsection{Annihilation juggling}
\label{subsec:juggle-annihilation}

As the reader is by now familiar with our approach, we define the
basic model and its enriched version at the same time.

\begin{Def} \label{Def:ANJMC} Given $h$ a nonnegative integer and
  $z_1,\ldots,z_{h+1}=a$ nonnegative real numbers such that
  $z_1+\cdots+z_h+a=1$, the (multivariate) \emph{annihilation model}
  is the Markov chain on the state space $\mathrm{St}_{h}$ for which
  the transition probability from $A=a_1a_2\ldots a_h$ to $B$ reads
  \begin{equation}
    P_{A,B} = \left\{
    \begin{array}{cl}
      z_i & \text{
        if $B=S_i(a_2\ldots a_h \circ)$ for some $i \in \{1,\ldots,k\}$,}\\
      z_{k+1} +\cdots + z_{h} + a & \text{if $B=a_2\ldots a_h \circ$,}\\
      0 & \text{otherwise,}
    \end{array} \right.
    \label{eq:anjmc}
  \end{equation}
  with $k$ the number of occurrences of $\circ$ in $a_2\ldots a_h
  \circ$. Similarly, the (multivariate) \emph{enriched
    annihilation model} is the Markov chain on the state space
  $\mathcal{S}(H)$ for which the transition probability from $\sigma$
  to $\tau$ is given by
  \begin{equation}
    \tilde{P}_{\sigma,\tau} = \left\{
      \begin{array}{cl}
        z_i & \text{
          if $\tau=J_i(\sigma^{\downarrow})$ for some $i \in \{1,\ldots,k\}$,}\\
        z_{k+1} +\cdots + z_{h} + a & \text{if $\tau=\sigma^\downarrow \cup \{H\}$,}\\
        0 & \text{otherwise,}
      \end{array} \right.
    \label{eq:anjmcplus}
  \end{equation}
  with $k$ the number of blocks of $\sigma^\downarrow$.
\end{Def}

\begin{rem}
  By our convention for $S_i(A)$ (resp.\ $J_i(\tau)$) when
  $i$ is larger than the number of occurrences of $\circ$ in $A$
  (resp.\ the number of blocks of $\tau$), we have the more compact
  expression
  $P_{A,B} = \sum z_i$ (resp.\ $\tilde{P}_{\sigma,\tau} = \sum z_i$)
  where the sum runs over all $i \in \{1,\ldots,h+1\}$ such that
  $B=S_i(a_2\ldots a_h \circ)$ (resp.\ $\tau=J_i(\sigma^\downarrow)$).
\end{rem}

\begin{rem}
  Warrington's annihilation model is recovered by taking
  $z_1=\cdots=z_h=a=1/(h+1)$. We still call our multivariate
  generalization the annihilation model, but this requires some
  clarification. Indeed, in Warrington's uniform case, one can
  interpret the dynamics by saying that ball insertions are made at
  arbitrary (empty or occupied) sites, and that a ball inserted at
  an occupied site is annihilated. However, in our multivariate
  generalization the correct interpretation is to say that we pick an
  $i$ between $1$ and $h+1$ and insert the ball at the $i$-th
  available site from the right, and that the ball is annihilated if
  there is no such site on the lattice.
\end{rem}

Here is the multivariate generalization of \cite[Theorem 3, item
2]{Warrington}:

\begin{thm} \label{thm:prob-juggle-annihilation} The stationary
  distribution of the annihilation model is given by
  \begin{equation}
    \label{eq:prob-juggle-annihilation} \Pi(B) = \prod_{\substack{i=1 \\
        b_i = \bullet}}^h 
    \left( z_1 + \cdots + z_{\psi_i(B)+1} \right)
    \prod_{j=1}^k (z_{j+1} + \cdots + z_h + a),
  \end{equation}
  for $B = b_1 \ldots b_h \in St_{h,k}$, with $\psi_i(B) = \# \{j:
  i<j\leq h, \, b_j = \circ\}$ as before. Similarly, the stationary
  distribution of the enriched annihilation model is given by
  \begin{equation}
    \label{eq:prob-juggle-annihilation-enriched}
    \tilde{\Pi}(\sigma) =  \prod_{(s,t) \text{ arch of } \sigma} \!\! z_{C_\sigma(s,t)} 
    \prod_{i=1}^{K-1} (z_{i+1} + \cdots + z_{H-1} + a),
  \end{equation}
  with $\sigma \in \mathcal{S}(H)$ and $K$ its number of blocks. There
  is no normalization factor, as $\Pi$ and $\tilde{\Pi}$ are already
  normalized for $z_1+\cdots+z_h+a=1$.
\end{thm}

\begin{exmp} \label{exmp:juggle-annihilation} The transition matrix of
  the basic annihilation model for $h=2$ in the
  lexicographically-ordered basis 
$(\bullet\bullet, \; \circ\bullet, \; \bullet\circ, \; \circ\circ)$  
 is given by
  \begin{equation}
    P = \left(
      \begin{array}{cccc}  \noalign{\medskip}z_1&0&z_{{2}}+a&0\\
        \noalign{\medskip}z_1&0&z_{{2}}+a&0\\
        \noalign{\medskip}0&z_{{1}}&z_{{2}}&a\\
       0&z_{{1}}&z_{{2}}&a
      \end{array} 
    \right) ,
  \end{equation}
  and admits the row vector $\left( {z_{{1}}^{2}} , \;
    z_{{1}} \left( z_{{2}}+a \right), \; \left( z_{{1}}+z_{{2}}
    \right) \left( z_{{2}}+a \right), a \left( z_{{2}}+a \right)\; \right)$ as a
  left eigenvector of eigenvalue $1$. The sum of the coordinates of
  this row vector is $(z_1+z_2+a)^2 = 1$. Similarly, the transition
  matrix of the enriched annihilation model for $H=3$ in the same
  ordered basis as in Example~\ref{exmp:setp-adddrop} is given by
  \begin{equation}
    \tilde{P} = \left(
      \begin {array}{ccccc} z_{{1}}&0&0&z_{{2}}+a&0\\
        \noalign{\medskip}z_{{1}}&0&0&z_{{2}}+a&0\\
        \noalign{\medskip}0&z_{{1}}&z_{{2}}&0&a\\
        \noalign{\medskip}0&z_{{1}}&z_{{2}}&0&a\\
        \noalign{\medskip}0&z_{{1}}&z_{{2}}&0&a
      \end {array} \right) ,
  \end{equation}
  and admits the row vector $\left( z_{1}^{2}, \;\; z_{{1}} \left(
      z_{{2}}+a \right), \;\; z_{{2}} \left( z_{{2}}+a \right), \;\;
    z_{{1}} \left( z_{{2}}+a \right), \;\; a\left( z_{{2}}+a \right)
  \right)$ as a left eigenvector of eigenvalue $1$.
\end{exmp}

Proving Theorem~\ref{thm:prob-juggle-annihilation} can be done by
checking ``by hand'' the stationarity of $\tilde{\Pi}$, then deducing
that of $\Pi$ using the projection property
\eqref{eq:mapintertwin}. However, the fact that $\tilde{\Pi}(\sigma)$
is not a monomial suggests that there exists a ``doubly enriched''
chain that yields both $P$ and $\tilde{P}$ by projection. It can be
seen that, if we relax the condition $z_1+\cdots+z_h+a=1$, then the
total mass of $\Pi$ (or of $\tilde{\Pi}$) is equal to
$(z_1+\cdots+z_h+a)^h$: this suggests that doubly enriched states
should consist of $h$-tuples of elements in a set of cardinality
$h+1$. From now on we will write $z_{h+1}$ in lieu of $a$.

\begin{Def}
  Given $h$ a nonnegative integer and $z_1,\ldots,z_{h+1}$ nonnegative
  real numbers such that $z_1+\cdots+z_{h+1}=1$, the \emph{doubly
    enriched annihilation model} is the Markov chain on the state
  space $\{1,\ldots,h+1\}^h$ for which the transition probability from
  $W=w_1 w_2\ldots w_h$ to $W'$ reads
  \begin{equation}
    \hat{P}_{W,W'} = \left\{
      \begin{array}{cl}
        z_i & \text{if $W'=w_2\cdots w_h i$,} \\
        0 & \text{otherwise.}
      \end{array} \right.
    \label{eq:anjmcplusplus}
  \end{equation}
\end{Def}

\begin{rem}
  The doubly enriched annihilation model can be seen a specialization
  of the de Bruijn process considered in \cite{AySt2011}.
\end{rem}

It is obvious that the stationary probability of $W=w_1\cdots w_h$
in the doubly enriched annihilation model is
\begin{equation}
  \label{eq:antristat}
  \hat{\Pi}(W) = z_{w_1} \cdots z_{w_h}
\end{equation}
since we are basically moving a window of size $h$ within a sequence
of independent identically distributed random variables. The
nontrivial fact is that this chain may be projected to the
annihilation and the enriched annihilation models.

\begin{thm}
  \label{thm:antrivproj}
  Let $\phi$ and $\tilde{\phi}$ be the mappings from $\bigcup_{j\geq0}
  \{1,\ldots,h+1\}^j$ to respectively $\bigcup_{j\geq0} \mathrm{St}_j$
  and $\bigcup_{j\geq0} \mathcal{S}(j+1)$ defined inductively by
  \begin{equation}
    \label{eq:anphidef}
    \phi(w_1 \cdots w_j)=
    \begin{cases}
      \emptyset & \text{for $j=0$,} \\
      S_{w_j}(\phi(w_1\cdots w_{j-1})\circ) & \text{for $j \geq 1$,}
    \end{cases}
  \end{equation}
  where $\emptyset$ denotes the empty word, and
  \begin{equation}
    \label{eq:anphitildef}
    \tilde{\phi}(w_1 \cdots w_j)=
    \begin{cases}
      \{\{1\}\} & \text{for $j=0$,} \\
      J_{w_j}(\tilde{\phi}(w_1\cdots w_{j-1})) & \text{for $j \geq 1$.}
    \end{cases}
  \end{equation}
  Then, the restrictions of $\phi$ and $\tilde{\phi}$ to
  $\{1,\ldots,h+1\}^h$ yield projections of the doubly enriched chain onto
  respectively the annihilation model and the enriched annihilation
  model.
\end{thm}

In other words, the projection $\phi$ (resp.\ $\tilde{\phi}$) consists
in applying the composition of the mappings $B \mapsto
S_{w_i}(B\circ)$ (resp.\ $\sigma \mapsto J_{w_i}(\sigma)$) with $i$
running from $1$ to $h$, to the ``seed'' $\emptyset$ (resp.\
$\{\{1\}\}$).

\begin{exmp}
 The transition matrix of
  the doubly enriched annihilation model for $h=2$ in the
 ordered basis $(11,21,31,12,22,32,13,23,33)$ is
  \begin{equation}
    \hat{P} = \left(
      \begin{array}{ccccccccc} 
      z_1 & 0 & 0 &z_2 & 0 & 0 & z_3 & 0  &0\\
     z_1 & 0 & 0 &z_2 & 0 & 0 & z_3 & 0  &0\\
     z_1 & 0 & 0 &z_2 & 0 & 0 & z_3 & 0  &0\\
      0 & z_1 & 0  & 0 & z_2 & 0 & 0 & z_3 & 0 \\
      0 & z_1 & 0  & 0 & z_2 & 0 & 0 & z_3 & 0 \\
      0 & z_1 & 0  & 0 & z_2 & 0 & 0 & z_3 & 0 \\
      0 & 0 & z_1 & 0 & 0 & z_2 & 0 & 0 & z_3\\
      0 & 0 & z_1 & 0 & 0 & z_2 & 0 & 0 & z_3\\
      0 & 0 & z_1 & 0 & 0 & z_2 & 0 & 0 & z_3\\
      \end{array} 
    \right).
  \end{equation}
  The Markov chains with transition matrices $\hat{P}$, $\tilde{P}$
  and $P$ for $h=2$ are displayed on Figure \ref{fig:markov_anni}.
  The words 12 and 13 project via $\tilde{\phi}$ to the set partition
  $3\,|\,1,2$ while 22 and 32 project to $1,3\,|\,2$, and all four project via
  $\phi$ to the juggling state $\bullet\circ$.
\end{exmp}

\begin{figure}[htpb]
  \centering
  \includegraphics[width=.7\textwidth]{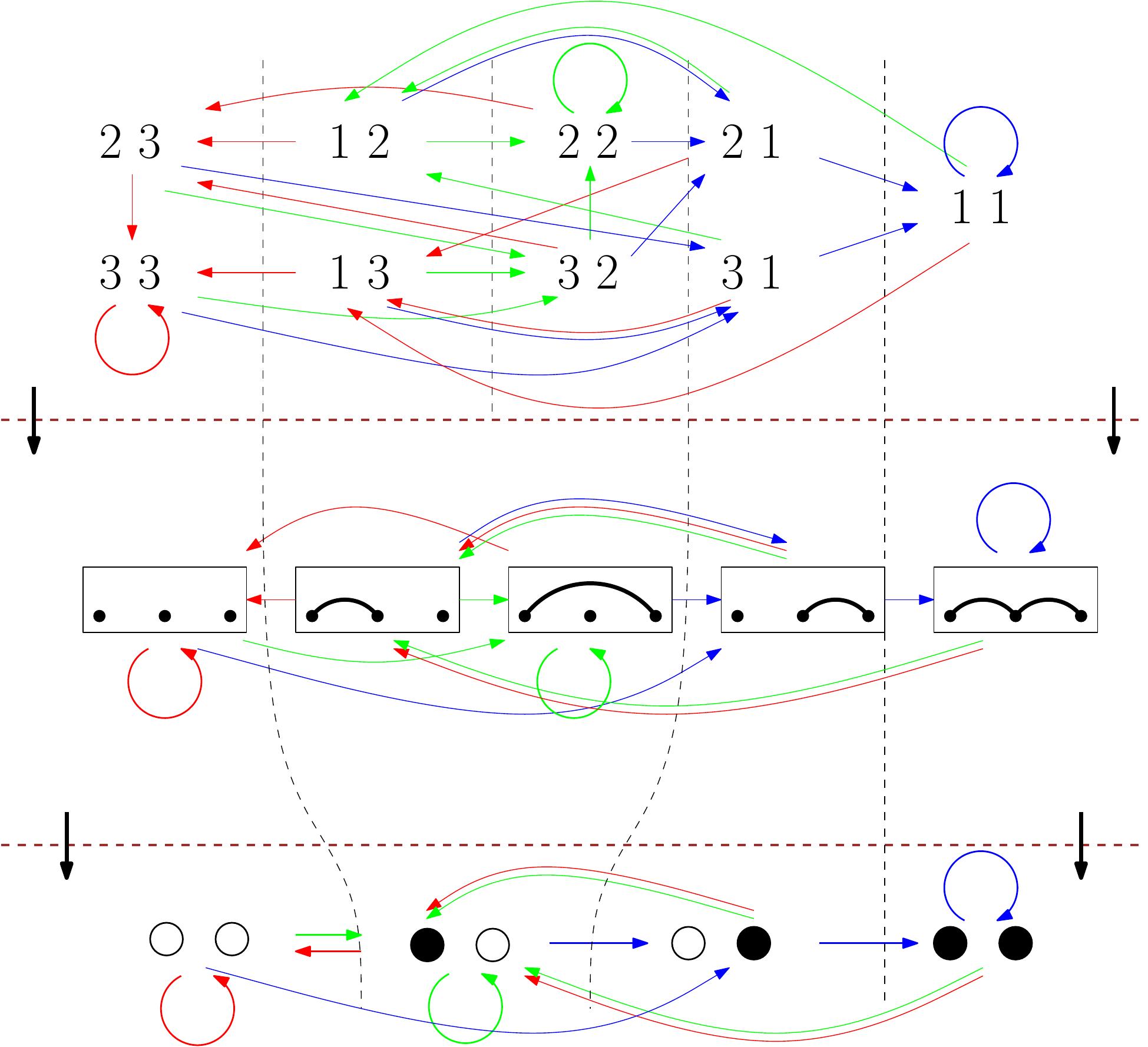}
  \caption{Transition graphs of the doubly enriched annihilation model
    (top), the enriched annihilation model (middle) and the
    annihiliation model (bottom) for $h=2$. Blue, green and red arrow
    represent transitions of respective probabilities $z_1,z_2,z_3$
    (when several arrows have the same endpoints, the corresponding
    probabilities should be added).}
  \label{fig:markov_anni}
\end{figure}

\begin{rem}
  The cardinality $h+1$ of the ``alphabet'' in the doubly enriched
  annihilation model is optimal in the sense that some states in
  $\mathrm{St}_h$ or $\mathcal{S}(h+1)$ are not attained with a
  smaller alphabet while, in a word of length $h$, all letters larger
  than $h+1$ can be replaced by $h+1$ without changing its image by
  $\phi$ or $\tilde{\phi}$. To be specific, if the alphabet has $\ell$
  letters with $\ell\le h$, the corresponding stationary probability
  of $B = b_1 \cdots b_h \in St_{h,k}$ in the annihilation model is
  \begin{equation}
    \Pi(B) = \prod_{\substack{i=1 \\
        b_i = \bullet}}^h 
    \left( z_1 + \cdots + z_{\psi_i(B)+1} \right)
    \prod_{j=1}^k (z_{j+1} + \cdots + z_\ell),
  \end{equation}
  if $k<\ell$, and 0 otherwise.  The corresponding stationary
  distribution of $\sigma \in \mathcal{S}(H,K)$ in the enriched
  annihilation model
  \begin{equation}
  \tilde{\Pi}(\sigma) =  \prod_{(s,t) \text{ arch of } \sigma} \!\! z_{C_\sigma(s,t)} 
    \prod_{i=1}^{K-1} (z_{i+1} + \cdots + z_{\ell})
  \end{equation}
  if $K\le \ell$, and 0 otherwise. When $\ell \geq h+1$, we might do
  the substitutions $z_{h+1} \leftarrow z_{h+1} + \cdots + z_{\ell}$ and
  $z_j \leftarrow 0$ for $j>h+1$ without affecting $\Pi$ and
  $\tilde{\Pi}$.
\end{rem}

\begin{proofof}{Theorem~\ref{thm:antrivproj}}
  It is not difficult to see that $\phi = \psi \circ \tilde{\phi}$
  from \eqref{eq:mapintertwin}; thus we need only check the statement
  for $\tilde{\phi}$. Consider the transition of probability $z_i$
  from $W=w_1 \cdots w_h$ to $W'= w_2 \ldots w_h i$ in the doubly
  enriched annihilation model: it is sufficient to have that
  $\tilde{\phi}(W)$ is sent to $\tilde{\phi}(W')$ by the corresponding
  transition of the enriched annihilation model, namely that
  \begin{equation}
    \label{eq:anwordphiWp}
    \tilde{\phi}(W') = J_i(\tilde{\phi}(W)^\downarrow).
  \end{equation}
  But this is an easy consequence from the commutation relation
  \eqref{eq:mapcommut} and the definition of $\tilde{\phi}$.
\end{proofof}

\begin{proofof}{Theorem~\ref{thm:prob-juggle-annihilation}}
  Let us first prove \eqref{eq:prob-juggle-annihilation-enriched}. We
  extend the definition $\eqref{eq:antristat}$ of $\tilde{\Pi}$ to
  words of arbitrary length by setting $\hat{\Pi}(w_1 \ldots w_j) =
  z_{w_1} \ldots z_{w_j}$. We will prove by induction on $j \leq h$
  that, for any $\sigma \in \mathcal{S}(j+1)$, we have
  \begin{equation}
    \label{eq:anrecsum}
    \sum_{W \in \tilde{\phi}^{-1}(\sigma)} \hat{\Pi}(W) =
    \prod_{(s,t) \text{ arch of } \sigma} \!\! z_{C_\sigma(s,t)} 
    \prod_{i=1}^{K-1} (z_{i+1} + \cdots + z_{h+1})
  \end{equation}
  with $K$ the number of blocks of $\sigma$. The relation is true for
  $j=0$ since both sides equal $1$. For $j \geq 1$, let $\sigma' \in
  \mathcal{S}(j)$ be the set partition obtained from $\sigma$ by
  removing the element $j+1$. The preimages by $\tilde{\phi}$ of
  $\sigma$ are then obtained from those of $\sigma'$ by
  \begin{itemize}
  \item appending a unique letter $w \leq K$ if $j+1$ is not a
    singleton in $\sigma$ (note that $\sigma$ has one arch more than
    $\sigma'$, and $w$ is the number of blocks that is covers),
  \item appending an arbitrary letter $w \geq K$ if $j+1$ is a
    singleton in $\sigma$ (note that $\sigma$ has one block more than
    $\sigma'$).
  \end{itemize}
  Then \eqref{eq:anrecsum} follows from the induction hypothesis: it
  is deduced from the relation for $\sigma'$ by adding an extra factor
  to the first product in the first case, and to the second product in
  the second case. We finally deduce
  \eqref{eq:prob-juggle-annihilation-enriched} by taking $j=h$ (recall
  that $H=h+1$ and $a=z_{h+1}$), noting that the left hand side of
  \eqref{eq:anrecsum} is nothing but the wanted $\tilde{\Pi}(\sigma)$.

  We then deduce \eqref{eq:prob-juggle-annihilation} from the relation
  $\phi=\psi \circ \tilde{\phi}$, using again Lemma~\ref{lem:preimag}.
\end{proofof}

Another nice property of the doubly enriched annihilation model 
which is easy to prove is that it is ``memoryless'': 
after $h$ transitions we end up with a
perfect sample of the stationary distribution, since all the initial
letters have been flushed out. 
Therefore, $h$ is a deterministic \emph{strong stationary time} of the doubly 
enriched chain \cite[Section 6.4]{LevinPeresWilmer}, which is independent
even of the initial distribution.
In other words, for any initial
probability distribution $\hat{\eta}$ over $\{1,\dots,h+1\}^h$, we
have
\begin{equation}
  \hat{\eta} \hat{P}^h = \hat{\Pi},
\end{equation}
and this implies that the only eigenvalues of the transition matrix
$\hat{P}$ are $1$ (with multiplicity $1$) and $0$. These properties
are clearly preserved by projection, which implies the following
nontrivial and remarkable property of the annihilation model, that
makes it distinct from the generic MJMC and add-drop models.

\begin{thm} \label{thm:spectrum-annihilation} For any initial
  probability distributions $\eta$ over $\mathrm{St}_{h}$ and
  $\tilde{\eta}$ over $\mathcal{S}(H)$, the distribution at time $h$
  is equal to the stationary distribution, namely
 \begin{equation}
   \label{eq:spectrum-annihilation}
   \eta P^h = \Pi, \qquad \tilde{\eta} \tilde{P}^h = \tilde{\Pi}.
 \end{equation}
 In particular, the only eigenvalues of $P$ and $\tilde{P}$ are $1$
 (with multiplicity $1$) and $0$.
\end{thm}

\begin{rem}
The fact that all other eigenvalues of the doubly enriched annihilation model are zero can also be proved directly from \cite[Theorem 12]{AySt2011} by setting $x_{a,m} = z_a$ so that $\beta_{a,m} = z_1 + \cdots + z_{h+1}$ for all $a \in [n], m \in [L]$ for words of length $L$ on an alphabet of size $n$. However, the fact that the distribution at time $L$ is the stationary distribution in general is not remarked in \cite{AySt2011}.
\end{rem}

This result is stronger than a statement about mixing times since we
reach the exact stationary state in bounded time! Another ``ultrafast
convergence to equilibrium'' was observed in \cite{LeskelaVarpanen}
for the model of Example~\ref{exmp:unbgeom}, but its combinatorial
origin seems rather different from that of the annihilation model (in
our model the time needed to reach the stationary distribution is
constant, while in Leskel\"a and Varpanen's model it depends on the
initial state, and can be arbitrarily large).

\section{Conclusion and discussion}
\label{sec:conclusion}
We end with speculative ideas for further work along the direction of
this paper.
All three annihilation Markov chains (on juggling sequences, on set partitions and on words) have the property that all the eigenvalues of the transition matrices are nonnegative and trivially linear in the parameters $z_i$. Such Markov chains typically arise from an underlying structure which we now describe. If one considers the matrices $P_i$ (resp. $\tilde{P}_i$) obtained by setting $z_i=1$ and $z_j=0$ for $j \neq i$ in the transition matrix $P$ (resp. $\tilde{P}$) in \eqref{eq:anjmc} (resp. \eqref{eq:anjmcplus}) 
then the
monoid generated by these matrices is 
$\mathscr R$-trivial. 
For a recent monograph on the connection between Markov chains and $\mathscr R$-trivial monoids, see \cite{asst2014}. One could reprove our statement about the eigenvalues in Theorem~\ref{thm:spectrum-annihilation} using this connection.
The general theory of free tree monoids described in \cite{asst2014}, however, does not directly apply to the annihilation chains because the generators $P_i$ (resp. $\tilde{P}_i$) do not always square to themselves (i.e. are not idempotents). It should be interesting to study when the results in \cite{asst2014} can be applied to situations when the generators are not idempotents.

This work suggests that it is worth studying a general framework for a class of problems in statistical physics, which we tentatively call ``boundary-driven Markov chains''. 
Just as the juggling and set-partition chains \eqref{eq:jugp} and \eqref{eq:jugp-enriched} studied in this paper, they have the property that the motion is deterministic for 
the most part (that is, in the bulk). Only when a ball (in the juggling context) reaches the boundary, i.e. Magnus' hand, something stochastic happens, namely the ball is thrown to a randomly chosen height. Similar things happen with the set-partition Markov chain. 

A natural idea is to combine the juggling process with an exclusion
process, for instance the Totally Asymmetric Exclusion Process (TASEP)
where particles do not all move at the same time but, instead, at each
time step one particle is selected randomly and moves if the next site
is empty. When a particle on the first site is selected, it is then
reinserted, say, uniformly at any available site on the lattice. This
model was investigated via a physical hydrodynamic approach in
\cite{AritaBouttierKrapivskyMallick}, where a nontrivial phase diagram
was found. It would be interesting to obtain more precise results on
this model, by looking for instance for a possible exact solution.

Finally, the fact that the stationary distributions of the Markov
chains considered in this paper all admit product forms is reminiscent
of the Zero-Range Process admitting a ``factorized stationary state''
\cite{EvansHanney,EvansMajumdar}, and one might wonder whether a
connection exists.

\noindent{\bf Acknowledgements:} The authors would like to thank N. Curien, P. Di Francesco, B. Haas, M. Josuat-Verg\`es and I. Kortchemski
for fruitful discussions during the completion of this work. The first author (A.A.) would like to acknowledge the hospitality during his stay at LIAFA, where this work was initiated.

\bibliographystyle{halpha}
\bibliography{juggle}

\end{document}